\documentclass[a4paper,11pt]{article}
\usepackage{xypic}
\usepackage{mathtools}
\usepackage{amsmath}
\usepackage{mathrsfs}
\usepackage{amssymb}
\usepackage{amsthm}
\usepackage{dsfont}
\usepackage{graphicx}
\usepackage{bmpsize}
\usepackage{times}
\usepackage{ textcomp }
\usepackage{color}
\usepackage{mathtools}
\usepackage{latexsym, empheq, fancybox}
\usepackage{mathrsfs}
\usepackage{exscale}
\usepackage{booktabs, array}
\usepackage[english]{babel}
\newcommand{\nc}{\newcommand}
\nc{\rnc}{\renewcommand}
\numberwithin{equation}{section}
\newtheorem{theorem}{Theorem}[section]
\newtheorem{lemma}[theorem]{Lemma}

\newtheorem{corollary}[theorem]{Corollary}

\newtheorem{example}[theorem]{Example}
\newtheorem{remark}[theorem]{Remark}
\usepackage{authblk}
\title{\Large\textbf{Hausdorff measure of sets of distributional chaotic pairs for shift maps }
}

\author{
Dalian Yuan
\thanks{Guangdong University of Education, Guangzhou 510303, P. R. China(E-mail: akqjok@qq.com)},
Ercai Chen
\thanks{School of Mathematical Sciences and Institute of Mathematics, Nanjing Normal University, Nanjing 210046, P. R. China,
and Center  of Nonlinear Science, Nanjing University, Nanjing 210093, P. R.China
(E-mail: ecchen@njnu.edu.cn)},\
\ \ and Zijie Lin
\thanks{School of Mathematical Sciences and Institute of Mathematics, Nanjing Normal University, Nanjing 210046, P. R. China
(E-mail: zjlin137@126.com)}
}

\begin{document}
\date{}
\maketitle

\begin{abstract}
Let $\sigma_{K}: \sum_{K}\rightarrow \sum_{K}$ be a shift map. For an interval $[p,q]\subset[0,1]$, let $D_{\sigma_{K}}([p,q])$ denote the set  of pairs for which the density spectrum of the $\epsilon$-approach time set equals $[p,q]$ when $\epsilon$ is small and $E_{\sigma_{K}}([p,q])$ the set of pairs for which the density spectrum of the $\epsilon$-approach time set converges to $[p,q]$ when $\epsilon\rightarrow 0^+$. Then $\dim_{H} D_{\sigma_{K}}([p,q])=\dim_{H} E_{\sigma_{K}}([p,q])=2-q$.
Moreover, $\mathscr{H}^{2-q}(E_{\sigma_{K}}([p,q]))=1$ when $q=0$
and $\mathscr{H}^{2-q}(E_{\sigma_{K}}([p,q]))=+\infty$ when $q>0$.
Meanwhile, $\mathscr{H}^{2-q}(D_{\sigma_{K}}([p,q]))=+\infty$ when $q=1$
and $\mathscr{H}^{2-q}(D_{\sigma_{K}}([p,q]))=0$ when $q<1$.

\end{abstract}

\section{Introduction}
The notion of chaos to describe the approaching-and-dispersing processes between trajectories in a dynamical system was first used in \cite{LY}.
Suppose $(X,\rho,f)$ is a { topological dynamical system} ({TDS} for short), namely, $(X,\rho)$ is a compact metric space and $f$ a continuous surjective self-map on $X$.
Then $(x,y)\in X\times X$ is said to be a { Li--Yorke pair} (\cite{BGKM}) if
\[
\liminf_{i\rightarrow \infty}(f^i(x),f^i(y))=0\ \text{and}\ \limsup_{i\rightarrow \infty}(f^i(x),f^i(y))>0.
\]
A set $C\subset X$ is called a (Li--Yorke) { scrambled set} if each pair of different points in $C$ forms a Li--Yorke pair.
In general, $f$ is said to be { Li--Yorke chaotic} if it has an uncountable scrambled set. It is proved in \cite{LY} that an interval map that has a periodic point of period $3$ is Li--Yorke chaotic. The existence of asymptotic pairs and Li--Yorke scrambled sets contained in the stable sets are studied in \cite{Huang2015}.

Based on Li--Yorke chaos, different types of chaos, such as Devaney
chaos (\cite{Dev}),
generic chaos (\cite{Pio}),
$\omega$-chaos (\cite{Li}), and Strong chaos (\cite{Xio1})
have been studied.

Distributional chaos, which was first introduced in \cite{SS} and was generalized in \cite{BSS}, \cite{PS} and \cite{PS1}, has been the focus of chaos study for more than ten years.
By describing the
densities of trajectory approach time sets, distributional chaos reveals more rigorous complexity hidden in Li--Yorke chaos.

We will now briefly review the definitions of the three types of  distributional chaos.
Let $(X,\rho,f)$ be a TDS.
For ${x},{y}\in X$, define the lower distributional function $F_{{x},{y}}$ and  upper distributional function $F_{{x},{y}}^*$ from $(0,+\infty)$ to $[0,1]$ by
\begin{equation} \label{deqiis}
\begin{aligned}
&F_{x, y}(\epsilon) =\liminf _{n \rightarrow \infty} \frac{1}{n} \#\left(\left\{0 \leq i<n : \rho\left(f^{i}(x), f^{i}(y)\right)<\epsilon\right\}\right), \\
&F_{x, y}^{*}(\epsilon) =\limsup _{n \rightarrow \infty} \frac{1}{n} \#\left(\left\{0 \leq i<n : \rho\left(f^{i}(x), f^{i}(y)\right)<\epsilon\right\}\right),
\end{aligned}
\end{equation}
where $\#(\cdot)$ denotes the cardinality of a set.
A couple $({x},{y})\in X\times X$ is called a DC1 { pair} if
\[
F_{x, y}^{*}(\epsilon) \equiv 1 \text { on }(0,+\infty) \text { and } F_{x, y}(\epsilon) \equiv 0 \text { on some }\left(0, \epsilon_{0}\right],
\]
a DC2 { pair} if
\[
F_{x,y}^*(\epsilon)\equiv 1 \text { on }(0,+\infty) \text { and } F_{{x},{y}}(\epsilon)<1 \text { on some }\left(0,\epsilon_{0}\right],
\]
and a DC3 { pair} if
\[
F_{{x},{y}}(\epsilon)<F_{{x},{y}}^*(\epsilon) \text{ on some } \left(\epsilon_{0},\epsilon_{1}\right ].
\]
A set $C\subset X$ is said to be a $\rm {DC}i$ ($i=1,2$ or 3) { scrambled set} if each pair of different points in $C$ forms a $\rm {DC}i$ pair. In general, $f$ is said to be $\rm {DC}i$ chaotic if it has an uncountable $\rm {DC}i$ scrambled set.

A pair $(x,y)\in X\times X$ is said to be a { mean Li--Yorke pair} if
$$
\liminf_{n\rightarrow\infty}\dfrac{1}{n}\sum\limits_{i=0}^{n-1}\rho(f^i(x),f^i(y))=0\text{ and } \limsup_{n\rightarrow\infty}\dfrac{1}{n}\sum\limits_{i=0}^{n-1}\rho(f^i(x),f^i(y))>0.
$$
A set $C\subset X$ is called a mean Li--Yorke chaotic set  if each pair of different points in $C$ forms a mean Li--Yorke pair.
In general, $f$ is said to be { mean Li--Yorke chaotic} if it has an uncountable mean Li--Yorke chaotic set. It is proved in \cite{Huang2014} that the intersections of the sets of asymptotic tuples and mean Li--Yorke tuples with the set of topological entropy tuples are dense in the set of topological entropy tuples.
It is observed in \cite{Dow} that DC2 chaos is equivalent to mean Li--Yorke chaos (see \cite{Dow} for details).

Cardinality (uncountable or not) is a simple description of the size of a scrambled set.
In fact, measures are widely used to characterize the sizes of scrambled sets.
The Lebesgue measures of scrambled sets are investigated in \cite{Smi1}, \cite{Smi2} and \cite{Mis}. The Bowen entropy dimensions of scrambled sets are studied in \cite{Hua} and \cite{FHYZ}.
The Hausdorff dimensions of strong scrambled sets and  DC1 scrambled sets are discussed in \cite{Xio1} and \cite{OS}, respectively.
 In the more recent paper \cite{BL}, the Lebesgue measure of Li--Yorke
pairs for interval maps is thoroughly discussed. 
In the present paper we study the Hausdorff measure of the set of distributional chaos pairs for shift maps.

A pair $(x,y)$ from a TDS is a DC1 pair if and only if the approach time sets of $x,y$ have upper density 1 and lower density 0, and is a DC2 pair if and only if the approach time sets have upper density 1 and lower density $<1$.
So the set of DC1 pairs and the set of DC2 pairs are saturated sets with diverging Birkhoff averages of approach time sets.
They are fractals generated by the distributional functions.
This viewpoint motivates the application of multifractal analysis to the study of chaos, so as to investigate distributional chaos in a more refined way than in terms of DC1 and DC2.

To give a more detailed description of our results, we introduce several definitions and notations. Let $\mathcal{C}([0,1])$ be the set of nonempty compact sub-intervals of ${[0,1]}$. Let $(X,\rho,f)$ be a TDS.
For $[p,q]\in\mathcal{C}([0,1])$, define
$$
\begin{aligned}
&E_f({[p,q]})=\{(x,y)\in X\times X:\lim_{\epsilon\rightarrow 0^{+}}F_{{x},{y}}^*(\epsilon)=q,\ \lim_{\epsilon\rightarrow 0^{+}}F_{{x},{y}}(\epsilon)=p\}, \\
&D_f({[p,q]})=\{(x,y)\in X\times X:F_{{x},{y}}^*(\epsilon)=q,\ F_{{x},{y}}(\epsilon)=p\ \text{on some}\ (0,\epsilon_{0}]\}.
\end{aligned}
$$
For $\mathcal{J}\subset\mathcal{C}([0,1])$, write
$$
E_{f}(\mathcal{J})=\bigcup_{I\in \mathcal{J}}E_{f}(I),\ D_{f} (\mathcal{J})=\bigcup_{I\in \mathcal{J} }D_{f}(I).
$$
The sets $E_f([p,q])$ and $[p,q]\in\mathcal{C}({[0,1]})$ form a spectral decomposition of the product space $X\times X$, while $E_f({[p,q]})$ and $D_f({[p,q]})$ are generalizations of the relations DC1 and DC2. In fact, for the map $f$, the relation DC1 equals $D_f({[0,1]})$ and the relation DC2 equals $ E_f({\{[p,1]:0\leq p<1\}})$.

For $[p,q]\in\mathcal{C}([0,1])$, we calculate the Hausdorff measures of the sets $E_{\sigma_K}({[p,q]})$ and $D_{\sigma_K}({[p,q]})$. They are as follows.
\begin{theorem} \label{t:main}
Let $[p,q]\in \mathcal{C}({[0,1]})$. Then
$$
\mathscr{H}^{2-q}(E_{\sigma_K}([p,q]))=\left\{
\begin{aligned}
&1,&q&=0,\\
&+\infty,&0&<q\leq 1
\end{aligned}
\right.
$$
and
$$
\mathscr{H}^{2-q}(D_{\sigma_K}([p,q]))=\left\{
\begin{aligned}
&0,&0&\leq q<1,\\
&+\infty,&q&=1.
\end{aligned}
\right.
$$
\end{theorem}

As an application, we get the following corollaries for the size of mean Li--Yorke chaos in symbolic space.
\begin{corollary}
Let $\rm{LY}(\sigma_K)$ be the set of all Li--Yorke pairs of the symbolic space $(\Sigma_K,\sigma_K).$ Then
 $\dim_{H}\rm{LY}(\sigma_K)=2,\mathscr{H}^2(\rm{LY}(\sigma_K))=1$.
\end{corollary}

\begin{corollary}
Let $\rm{MLY}(\sigma_K)$ be the set of all mean Li--Yorke pairs of the symbolic space $(\Sigma_K,\sigma_K).$ Then
 $\dim_{H}\rm{MLY}(\sigma_K)=1,\mathscr{H}^1(\rm{MLY}(\sigma_K))=+\infty$.

\end{corollary}
The main body of this paper is organized as follows.

In Section \ref{desooi}, some necessary definitions and notations are specified.

In Section \ref{desoow}, the distributional functions $\mathcal{F}_{f}$ and $\mathcal{E}_{f}$ are defined. The distributional chaotic relations $E_{f}({[p,q]})$, $D_{f}({[p,q]})(f)$, $E_{f} ({\mathcal{J}})$ and $D_{f}({\mathcal{J}})$ are introduced. Some properties of certain invariances of these relations are discussed.

Section \ref{desoou} is a review of the basic properties of Hausdorff measure on symbolic spaces. Some useful lemmas are proved.

In Section \ref{desoos}, we give an useful variational inequality for calculating the Hausdoff dimensions and Hausdorff measures of the sets $E_{\sigma_K}([p,q])$ and $D_{\sigma_K}([p,q])$.

In Section \ref{desooz}, we study the Hausdorff dimensions of $E_{\sigma_K}({\mathcal{J}})$ and $D_{\sigma_K}(\mathcal{J})$. It is proved that for $\emptyset\neq\mathcal{J}\subset\mathcal{C}({[0,1]})$,
\begin{equation} \label{deqivi}
\dim_H E_{\sigma_K}({\mathcal{J}})=\dim_H  D_{\sigma_K}({\mathcal{J}})=2-\inf\{\sup I:{I\in\mathcal{J}}\}.
\end{equation}
In particular, for $[p,q]\in\mathcal{C}({[0,1]})$,
\begin{equation} \label{q1203141204}
\dim_H E_{\sigma_K}({[p,q]})=\dim_H  D_{\sigma_K}({[p,q]})=2-q.
\end{equation}
From (\ref{deqivi}) we have that the Hausdorff dimension of the set of DC1 (or DC2) pairs  for $\sigma_K$ is 1. We will prove Corollary 1.2.

Section \ref{desooe} is the proof of Theorem \ref{t:main}.

\section{Some definitions and notations}\label{desooi}
For a number $a$ and sets of numbers $B,C$,  we make use of the following notation:
$$
\begin{aligned}
&a+B=B+a=\{a+b:b\in B\},\ aB=Ba=\{ab:b\in B\},\\
&B+C=\{b+c:b\in B,\ c\in C\},\ BC=\{bc:b\in B,\ c\in C\}.\\
\end{aligned}
$$
When $X$ is a set, $\mathcal{P}(X)$ denotes the power set of $X$. For $A\subset X$, $A^c$ denotes the complement of $A$, $X\setminus A$. We use $\Delta=\Delta(X)$ to denote the diagonal $\{(x,x):x\in X\}$ in $X\times X$.
In this paper we use $\rho$ to denote any metric. Suppose that $X$ is a metric space. For $\epsilon>0$ we use $\Delta_\epsilon$ to denote the set $\{(x,y)\in X\times X:\rho(x,y)<\epsilon\}$. For  $x\in X$ and nonempty sets $A,B\subset X$, define
$$
\begin{aligned}
&B_\epsilon(x)=\{y\in X:\rho(x,y)<\epsilon\},\ B_\epsilon(A)=\bigcup_{y\in A}B_\epsilon(y),\\
&\rho(A,B)=\inf\{\rho(a,b):a\in A,\ b\in B\},\ \rho(x,A)=\rho(A,x)=\rho(\{x\},A).
\end{aligned}
$$
We use $|\cdot|$ to denote the diameter of a set.

Let $Y=\prod_{0\leq i<\alpha} X_i$, where $\alpha$ is an ordinal number $\leq\omega_{0}$. If $x\in Y$, we use $x_j$ to denote the $(j+1)$th coordinate of $x$. If $(y_i)$ is a sequence in $Y$, we use $y_{i,j}$ to denote the $(j+1)$th coordinate of $y_i$, i.e., $y_{i,j}=(y_i)_j$.

For a product $Y=\prod_{0\leq i<n}X_i$ of finitely many metric spaces, unless otherwise specified, we endow $Y$ with the sup metric
$$
\rho(x,y)=\sup_{0\leq i<n}\rho(x_i,y_i).
$$
Suppose $X$ is a nonempty separable metric space. We use $\mathcal{C}(X)$ to denote the set of nonempty compact connected subsets of $X$. For a sequence $\alpha=(x_n)_{n\geq 0}$ of points in $X$, we use $\omega(\alpha)$ to denote the set of limit points of $\alpha$ set, i.e.,
$$
\omega(\alpha)=\{x\in X: \text{for each neighborhood}\ U of\ x, x_{n}\in U \text{ for infinitely many}\ n\}.
$$

\begin{lemma}\label{detvxv}
Let $X$ be a nonempty compact metric space. Suppose $(x_i)_{i\geq 0}$ is a sequence of points in $X$ with $\lim_{i\rightarrow \infty}\rho(x_i,x_{i+1})=0$. Then $\omega(x_i:i\geq0)\in\mathcal{C}(x)$.
\end{lemma}

We omit the proof of Lemma \ref{detvxv}, for it is easy.

Let $\alpha=({n_i})_{i\geq 0}$ be a sequence of positive integers with infinitely many $n_i\geq 2$. Write
$$\Sigma_\alpha=\prod_{i\geq 0}\{0,\cdots,n_{i}-1\}=\{(x_i)_{i\geq 0}:x_i\in\{0,\cdots,n_{i}-1\},\ i\geq 0\}.$$
For $x,y\in\Sigma_\alpha$, write
$$
\delta(x,y)=\inf \{i\geq0:x_i\neq y_i\},
$$
where $\delta(x,x)=+\infty$.
Endow $\Sigma_\alpha$ with the metric
$$
\rho(x,y)=\prod_{0\leq i<\delta(x,y)}n_i^{-1}.
$$
Write
$$
W_{\alpha,i}=\prod_{0\leq j<i}\{0,\cdots,n_{j}-1\},\ W_\alpha=\bigcup_{i\geq 0}W_{\alpha,i}.
$$
If $\omega\in W_{\alpha,i}$, then $\omega$ is called a {\bf word} with {\bf length} $|\omega|=i$. Write
$$
[\omega]=\left\{x\in\Sigma_\alpha:x_{0}\cdots x_{i-1}=\omega\right\},
$$
where $x_{0}\cdots x_{i-1}$ is the concatenation of letters $x_{0}\cdots x_{i-1}$. In addition, $[\omega]$ is said to be a {\bf cylinder} in $\Sigma_\alpha$ of {\bf length} $i$.
For $W\subset W_\alpha$, write
$$
[W]=\bigcup\left\{[\omega]:\omega\in W_\alpha\right\}.
$$
Suppose each $n_i=k$. Then we write $\Sigma_k$, $W_{k,i}$, $W_k$ for $\Sigma_\alpha$, $W_{\alpha,i}$, $W_\alpha$ respectively.

The properties stated in the lemma below are direct.

\nc{\detozi}{Lemma \ref{detozi}}
\begin{lemma}\label{detozi}
Let $\alpha=({n_i})_{i\geq 0}$  be a sequence of positive integers with infinitely many $n_i\geq2$.
Then $\Sigma_\alpha=\prod_{i\geq 0}\{0,\cdots,n_{i}-1\}$ is a Cantor space.
For each word $\omega\in W_\alpha$, the cylinder $[\omega]$ is closed and open with diameter $|[\omega]|=\prod_{0\leq i<|\omega|} n_i^{-1}$.
The set $\{[\omega]:\omega\in W_\alpha\}\cup\{\emptyset\}$ is a base of the topology of $\Sigma_\alpha$.
\end{lemma}

Throughout this paper, $K\geq2$ denotes a fixed natural number.
By Lemma \ref{detozi}, $\Sigma_K$ is a Cantor space, each cylinder $[\omega]$ is a closed and open subset of $\Sigma_K$ with diameter $K^{-|\omega|}$, and the set $\{[\omega]:\omega\in W_K\}\cup\{\emptyset\}$ is a topological base for $\Sigma_K$.
Define the {\bf shift map} $\sigma_K$ on $\Sigma_K$ as the map $(\sigma_{K}(x))_i=x_{i+1}$, $i\geq 0$.
It is a $K$ to $1$ continuous map.

Suppose $\emptyset\neq A_i\subset W_{K,n_{i}}$, $i\geq 0$. We write
$$
\prod_{i\geq 0} A_i=\left\{\omega_{0}\omega_{1}\cdots \in\Sigma_K: w_i\in A_j,\ j\geq 0\right\}.
$$
where $\omega_{0}\omega_{1}\cdots$ is the concatenation of words $\omega_{0},\omega_{1},\cdots$.

Let $N\subset\mathbb{N}$. For $n\geq1$, write $\zeta_n(N)=\#(N\cap\{0,\cdots,n-1\})$ and $\mu_n(N)=\frac{\zeta_{n}(N)}{n}$.
Put $\mu(N)=\omega(\mu_n(N):n\geq1)$ and call it the $\mathbf{density\ spectrum}$  of $N$.
By Lemma \ref{detvxv}, $\mu(N)$ is a nonempty subinterval of $[0,1]$, i.e., $\mu(N)\in\mathcal{C} ({[0,1]})$.
We call $\mu_*(N):=\inf \mu(N)$ the $\mathbf{lower \ density}$ of $N$ and $\mu^*(N):=\sup \mu(N)$ the $\mathbf{upper\ density}$ of $N$.
When $\mu(N)=[p,p]$, we also write $\mu(N)=p$.

Define a partial order $\preceq$ on $\mathcal{C}({[0,1]})$ by
$$I\preceq J\leftrightarrow\inf I\leq\inf J\ \text{and}\ \sup I\leq\sup J.$$

The properties stated in the lemma below are direct.

\begin{lemma}\label{dstoxm}
Let $N$, $M\subset\mathbb{N}$. Write $N=\{n_i:0\leq i<\#( N)\}$ with $n_i<n_{i+1}$.
\begin{enumerate} \rnc{\labelenumi}{(\alph{enumi})}
\item If $M\subset N$, then $\mu(M)\preceq\mu(N)$.
\item If $\liminf_{i\rightarrow\infty}(n_{i+1}-n_i)\geq k\geq1$, then $$\mu(N+\{0,\cdots,k-1\})=\mu((N-\{0,\cdots,k-1\})\cap\mathbb{N})=k\mu N.$$
    So, if $\lim_{i\rightarrow\infty}(n_{i+1}-n_i)=+\infty$, then $\mu (N)=0$.
\item $\mu(kN+\{0,\cdots,k-1\})=\mu((kN-\{0,\cdots,k-1\})\cap\mathbb{N})=\mu (N)$ for $k\geq 1$.            $\hfill \Box$
\end{enumerate}
\end{lemma}

\section{Distributional functions and distributional chaotic relations}\label{desoow}

Let  $(X,\rho,f)$ be a TDS. For $x\in X$ and $A\subset X$, define the {\bf recurrence time set} $N_f(x,A)$ by
$$N_f(x,A)=\{i\geq 0:f^i(x)\in A\}.$$
Define $\mathcal{F}_{f} :(X\times X)\times(0,+\infty)\rightarrow\mathcal{C}({[0,1]})$ by

\begin{equation} \label{deqozw}
\mathcal{F}_{f} ((x,y),\epsilon)=\mu({N_{f\times f}((x,y),\Delta_{\epsilon})}).
\end{equation}
For $(x,y)\in X\times X$,

\begin{equation} \label{deqozz}
\begin{aligned}
0<\epsilon_{0}<\epsilon_1&\Rightarrow N_{f\times f}((x,y),\Delta_{\epsilon_{0}})\subset N_{f\times f}((x,y),\Delta_{\epsilon_{1}})\\
&\Rightarrow
\mathcal{F}_{f}((x,y),\epsilon_0)\preceq \mathcal{F}_{f}((x,y),\epsilon_1).
\end{aligned}
\end{equation}
By (\ref{deqozz}), we define $ \mathcal{E}_f:X\times X\rightarrow\mathcal{C}({[0,1]})$ by

\begin{equation} \label{deqozx}
\mathcal{E}_f(x,y)=[\lim_{\epsilon\rightarrow 0^+}\inf\mathcal{F}_{f}((x,y),\epsilon),\lim_{\epsilon\rightarrow 0^+}\sup\mathcal{F}_{f}((x,y),\epsilon)].
\end{equation}
For $[p,q]\in\mathcal{C}({[0,1]})$. Write
$$
E_{f}({[p,q]})=\{(x,y)\in X\times X:\mathcal{E}_{f}({x,y})=[p,q]\}
$$
and define
$$
D_{f}({[p,q]})=\{(x,y)\in X\times X:\mathcal{F}_{f}((x,y),\epsilon)\equiv [p,q]\ \text{on some interval}\ (0,\epsilon_0]\}.
$$
Note that $D_{f}({[p,q]})\subset E_{f}({[p,q]})$.
For $\mathcal{J}\subset\mathcal{C}({[0,1]})$, put
$$
E_f(\mathcal{J})=\bigcup_{I\in\mathcal{J}}E_{f}(I),\ D_{f}(\mathcal{J})=\bigcup_{I\in\mathcal{J}}D_{f}(I).
$$

\begin{remark} \label{dstios}
{\rm The distributional chaotic relation with respect to DC1 is $D_f({[0,1]})$ and the distributional chaotic relation with respect to DC2 is $E_f({\{[p,1]:0\leq p<1\}})$.}
\end{remark}

\begin{lemma}\label{detoss}
Let $(X,f)$ be a TDS and $n\geq1$. Then
$$
\mathcal{E}_{f^n}=\mathcal{E}_f
$$
and, for ${[p,q]}\in\mathcal{C}({[0,1]})$,
$$
E_{f^n}({[p,q]})=E_{f}({[p,q]})\ \text{and}\ D_{f^n}({[p,q]})=D_{f}({[p,q]}).
$$
\end{lemma}

\begin{proof}
By the uniform continuity of $f$, we may choose positive numbers $\epsilon_i\rightarrow0^+$ such that, for each $(x,y)\in X\times X$ and $i\geq0$,

\begin{equation} \label{deqoos}
\rho({x,y})<\epsilon_{i+1}\Rightarrow \rho({f^j(x),f^j(y)})<\epsilon_i\ \text{for}\ 0\leq j<n.
\end{equation}

Let $(x,y)\in X\times X$. We are to verify

\begin{equation} \label{deqoox}
nN_{f^n\times f^n}(({x,y}),\Delta_{\epsilon_{i+1}})+\{0,\cdots,n-1\}\subset N_{f\times f}(({x,y}),\Delta_{\epsilon_i})
\end{equation}
and
\begin{equation} \label{deqoou}
(nN_{f^n\times f^n}((x,y),\Delta_{\epsilon_i}^c)-\{0,\cdots,n-1\})\cap\mathbb{N}\subset N_{f\times f}((x,y),\Delta_{\epsilon_{i+1}}^c).
\end{equation}
Suppose $i\in N_{f^n\times f^n}((x,y),\Delta_{\epsilon_{i+1}})$ and $0\leq j<n$.
Then $\rho(f^{in}(x),f^{in}(y))<\epsilon_{i+1}$.
By (\ref{deqoos}), $\rho(f^{in+j}(x),f^{in+j}(y) )<\epsilon_{i}$, which means $in+j\in N_{f\times f}((x,y),\Delta_{\epsilon_i})$.
So (\ref{deqoox}) holds.
Suppose $i\geq0$ satisfies $i\in N_{f^n\times f^n}((x,y),\Delta_{\epsilon_i}^c)$, $0\leq j<n$ and $in-j\geq0$.
Now $\rho( {f^{in}(x),f^{in}(y)})\geq\epsilon_{i}$. By (\ref{deqoos}), $\rho(f^{in-j}(x),f^{in-j}(y))\geq\epsilon_{i+1}$, which means $in-j\in N_{f\times f}((x,y),\Delta_{\epsilon_{i+1}}^c)$. So (\ref{deqoou}) holds.

Eq\. (\ref{deqoou}) leads to

\begin{equation} \label{deqooe}
\begin{aligned}
&N_{f\times f}((x,y),\Delta_{\epsilon_{i+1}})\\
=&(N_{f\times f}((x,y),\Delta_{\epsilon_{i+1}}^c))^c\\
\subset&((nN_{f^n\times f^n}((x,y),\Delta_{\epsilon_i}^c)-\{0,\cdots,n-1\}\cap\mathbb{N})^c\\
=&\mathbb{N}\setminus(nN_{f^n\times f^n}((x,y),\Delta_{\epsilon_i}^c)-\{0,\cdots,n-1\})\\
=&(n(N_{f^n\times f^n}((x,y),\Delta_{\epsilon_i}^c))^c-\{0,\cdots,n-1\})\cap\mathbb{N}\\
=&(nN_{f^n\times f^n}((x,y),\Delta_{\epsilon_i})-\{0,\cdots,n-1\})\cap\mathbb{N}.
\end{aligned}
\end{equation}
Eqs (\ref{deqoox}) and (\ref{deqooe}) lead to

\begin{equation} \label{qivioveviuo}
\begin{aligned}
&nN_{f^n\times f^n}( ({x,y}),\Delta_{\epsilon_{i+2}})+\{0,\cdots,n-1\}\\
\subset&N_{f\times f}((x,y),\Delta_{\epsilon_{i+1}})\\
\subset&(nN_{f^n\times f^n}((x,y),\Delta_{\epsilon_i})-\{0,\cdots,n-1\})\cap\mathbb{N}.
\end{aligned}
\end{equation}
Applying (c) of Lemma \ref{dstoxm} to (\ref{qivioveviuo}), we get

\begin{equation} \label{deqooz}
\mathcal{F}_{f^n}((x,y),\epsilon_{i+2})\preceq \mathcal{F}_{f}((x,y),\epsilon_{i+1})\preceq \mathcal{F}_{f^n}((x,y),\epsilon_{i}).
\end{equation}
Letting $i\rightarrow\infty$ in (\ref{deqooz}) we obtain

\begin{equation} \label{deqons}
\mathcal{E}_{f^n}(x,y)=\mathcal{E}_{f}(x,y).
\end{equation}
So $\mathcal{E}_{f^n}=\mathcal{E}_f$ and, for $[p,q]\in\mathcal{C}({[0,1]})$,
$$
E_{f^n}({[p,q]}) =\mathcal{E}_{f^n}^{-1}([p,q])=\mathcal{E}_{f}^{-1}([p,q])=E_{f}({[p,q]}).
$$

For $(x,y)\in X\times X$ and $\epsilon>0$, put

\begin{equation} \label{deqwov}
\begin{aligned}
&\mathcal{G}_{f,\epsilon}(x,y)\\
=&[\inf \mathcal{F}_f((x,y),\epsilon)-\inf \mathcal{E}_f(x,y),\sup \mathcal{F}_f((x,y),\epsilon)-\sup \mathcal{E}_f(x,y)]\\
\in&\mathcal{C} ({[0,1]}).
\end{aligned}
\end{equation}
Then, by (\ref{deqooz}) and (\ref{deqons}), we have

\begin{equation} \label{deqoio}
\mathcal{G}_{f^n,\epsilon_{i+2}}(x,y)\preceq \mathcal{G}_{f,\epsilon_{i+1}}(x,y)\preceq \mathcal{G}_{f^n,\epsilon_{i}}(x,y).
\end{equation}
So

\begin{equation} \label{deqozn}
\text{if}\ \mathcal{G}_{f,\epsilon_{i}}(x,y)=[0,0]\ \text{then}\ \mathcal{G}_{f^n,\epsilon_{i+1}}(x,y)=[0,0]
\end{equation}
and

\begin{equation} \label{deqiio}
\text{if}\ \mathcal{G}_{f^n,\epsilon_{i}}(x,y)=[0,0]\ \text{then}\ \mathcal{G}_{f,\epsilon_{i+1}}(x,y)=[0,0].
\end{equation}
Now (\ref{deqozn}), (\ref{deqiio}), (\ref{deqwov}) and (\ref{deqons}) imply
$$
(x,y)\in D_{f^n}({[p,q]})\Leftrightarrow (x,y)\in D_{f}({[p,q]}).
$$
Then $D_{f^n}({[p,q]})=D_f({[p,q]})$.
\end{proof}

\section{Hausdorff measure on symbolic spaces}\label{desoou}

We will now briefly review the concept of Hausdorff measure and Hausdorff dimension. See \cite{Fal} for more details. Let $X$ be a separable metric space.
Then $\mathcal{A}\subset\mathcal{P}(X)$ is called a {\bf cover} of $X$ if $\bigcup\mathcal{A}=X$. A cover $\mathcal{A}$ with $|\mathcal{A}|:=\sup_{A\in\mathcal{A}}|A|<\delta$ is called a {\bf $\delta$-cover}. Let $\mathscr{C}(X)$ denote the set of countable covers of $X$ and $\mathscr{C}(X,\delta)$ the set of countable $\delta$-covers of $X$.
For $0\leq s\leq+\infty$ and $\delta>0$, define the {\bf $\mathscr{H}_\delta^s$ measure} of $X$ as
$$
\mathscr{H}_\delta^s(X)=\inf_{\mathcal{A}\in\mathscr{C}(X,\delta)}\sum_{A\in\mathcal{A}}|A|^s,
$$
where $0^0=0^{+\infty}=0$. It is obvious that $\mathscr{H}_\delta^s(X)$ is nondecreasing while $\delta$ decreases. Then define the {\bf $s$-dimensional Hausdorff measure} of $X$ as
$$
\mathscr{H}^s(X)=\lim_{\delta\rightarrow 0^+}\mathscr{H}_\delta^s(X)=\sup_{\delta>0}\mathscr{H}_\delta^s(X)\in[0,+\infty].
$$
For $0\leq s<t$, by
$$
\mathscr{H}_\delta^s(X)\geq\delta^{s-t}\mathscr{H}_\delta^t(X),\ 0<\delta<1,
$$
if $\mathscr{H}^t(X)>0$, then $\mathscr{H}^s(X)=+\infty$. Then there is a unique value $\dim_{H} X\in[0,+\infty]$, called the {\bf Hausdorff dimension} of $X$, such that
$$
0\leq s<\dim_{H} X\Rightarrow \mathscr{H}^s(X)=+\infty\ \text{and}\ \dim_{H} X<s\leq+\infty\Rightarrow \mathscr{H}^s(X)=0.
$$

The next two lemmas are well known.

\begin{lemma}\label{detoex}
Let $X$ be a separable metric space and $\mathcal{A}$ a countable set of subsets of $X$. Then, for $0\leq s\leq+\infty$,
\begin{equation} \label{deqvuu}
\sup_{A\in\mathcal{A}}\mathscr{H}^s(A)\leq\mathscr{H}^s\left(\bigcup\mathcal{A}\right)\leq\sum_{A\in\mathcal{A}}\mathscr{H}^s(A).
\end{equation}
Thus
\begin{equation} \label{deqvus}
\dim_{H}\bigcup\mathcal{A}=\sup_{A\in\mathcal{A}}\dim_{H} A.\ \Box
\end{equation}
\end{lemma}
\begin{lemma}\label{detviw}
Suppose $X,Y$ are separable metric spaces and $\pi$ is a surjective map from $X$ to $Y$. Let $s,c\in(0,+\infty)$. If for some $\delta_0>0$,
$$
\rho(\pi(x),\pi(y))\leq c(\rho(x,y))^s\ \text{while}\ x,y\in X\ \text{with}\ 0<\rho(x,y)<\delta_0,
$$
then
$$
\mathscr{H}^t(Y)\leq c^t\mathscr{H}^{st}(X)\ \text{for}\ t\in(0,+\infty),
$$
and thus
$$
\dim_H Y\leq \frac{1}{s}\dim_H X.\ \Box
$$
\end{lemma}

\begin{lemma}\label{detouu}
Let $X$, $Y$ be separable metric spaces. Suppose $\pi:X\rightarrow Y$ is surjective with
$$
\lim_{\delta\rightarrow 0^+}\inf\left\{\frac{\ln\rho( {\pi(x),\pi(y)})}{\ln\rho (x,y)}:(x,y)\in X\times X,\ 0<\rho(x,y)<\delta\right\}=s.
$$
Then
\begin{equation} \label{deqouv}
s\dim_HY\leq\dim_HX.
\end{equation}
So, if $s>0$, or $s\geq0$ and $\dim_HX>0$, then
\begin{equation} \label{deqouw}
\dim_HY\leq\frac{1}{s}\dim_HX.
\end{equation}
\end{lemma}

\begin{proof}
If $s\leq0$, the inequality (\ref{deqouv}) is obvious. Then suppose $s>0$. Let $0<\tau<s$. Pick $0<\delta_0<1$ such that
\nc{\deqoiv}{(\ref{deqoiv})}
$$
\begin{aligned}
\frac{\ln\rho( {\pi(x),\pi(y)})}{\ln\rho (x,y)}\geq\tau,\ &\text{i.e.}\ \rho(\pi(x),\pi(y))\leq(\rho(x,y))^\tau,\\
&\text{for}\ (x,y)\in X\times X\ \text{and}\ 0<\rho(x,y)<\delta_0.
\end{aligned}
$$
By Lemma(\ref{detviw}), $\dim_H Y\leq\frac{1}{\tau}\dim_H X$, i.e. $\tau\dim_HY\leq\dim_HX$. Letting $\tau\nearrow s$ we get (\ref{deqouv}).
\end{proof}

\begin{lemma}\label{detouw}
Let $\alpha=({n_i})_{i\geq 0}$ be a sequence of positive integers with infinitely many $n_i\geq2$. Let $\Sigma_\alpha=\prod_{i\geq 0}\{0,\cdots,n_{i}-1\}$. Then $\dim_H\Sigma_\alpha=1$ and $\mathscr{H}^1(\Sigma_\alpha)=1$.
\end{lemma}

\begin{proof}
Let $\delta>0$. Choose $k$ with $\prod_{0\leq i\leq k}n_i^{-1}<\delta$. Then $\{[\omega]:\omega\in W_{\alpha,k}\}$ is a finite $\delta$-cover of $\Sigma_\alpha$ with
$$
\sum_{\omega\in W_{\alpha,k}}|[\omega]|=\prod_{0\leq i<k} n_i\cdot\prod_{0\leq i<k}n_i^{-1}=1.
$$
Since $\delta>0$ was arbitrary, we have $\mathscr{H}^1({\Sigma_\alpha})\leq1$ and thus $\dim_H\Sigma_\alpha\leq1$.

Let $(A_i)_{i\geq0}$ be a countable cover of $\Sigma_\alpha$. We are to show

\begin{equation} \label{deqoui}
\sum_{i\geq0}|A_i|\geq1.
\end{equation}
Let $\epsilon>0$. For each $i$, let $B_i$ be a cylinder containing $A_i$ with
$$
|B_i|<|A_i|+2^{-i}\epsilon.
$$
Now $(B_i)$ is a sequence of closed and open sets covering $\Sigma_\alpha$. Since $\Sigma_\alpha$ is compact and each $B_i$ is open, we can choose a finite subcover $(B_{i_{j}})_{0\leq j< k}$. Suppose $B_{i_{j}}=[\omega_{i_{j}}]$, $0\leq j< k$. Let $l_{i_{j}}$ be the length of $\omega_{i_{j}}$ and put $l=\sup_{0\leq j< k}l_{i_{j}}$. Then
$$
\begin{aligned}
|[\omega_{i_j}]|&=\prod_{0\leq i<l_{i_j}}n_i^{-1}\\
&=\prod_{l_{i_j}\leq i<l}n_i\cdot\prod_{0\leq i<l}n_i^{-1}\\
&=\sum\{|[\omega]|:\omega\in W_{\alpha,l},\ \omega|_{\{0,\cdots,l_{i_j}-1\}}=\omega_{i_j}\}.
\end{aligned}
$$
Now
$$
\begin{aligned}
\sum|A_i|&\geq\sum|B_i|-2\epsilon\\
&\geq\sum_{0\leq j<k}|B_{i_j}|-2\epsilon\\
&=\sum_{0\leq j<k}|[\omega_{i_j}]|-2\epsilon\\
&\geq\sum\{|[\omega]|:\omega\in W_{\alpha,l}\}-2\epsilon\\
&=\prod_{0\leq i<l}n_i\cdot\prod_{0\leq i<l}n_i^{-1}-2\epsilon\\
&=1-2\epsilon.
\end{aligned}
$$
Since $\epsilon>0$ was arbitrary, we have $\sum|A_i|\geq1$. Since $(A_i)$ was arbitrary, we have $\mathscr{H}^{1}{(\Sigma_\alpha)}\geq1$ and thus $\dim_H\Sigma_\alpha\geq1$.
\end{proof}

\begin{lemma}\label{detoue}
Suppose, for $i\geq 0$, that $\emptyset\neq A_i\subset W_{K,{n_i}}$ and $\#( A_i)=a_i$. Then
\begin{equation} \label{deqwow}
\dim_H\prod_{i\geq 0} A_i\geq\liminf_{j\rightarrow \infty}\frac{\sum_{0\leq i<j}\ln a_{i}}{\sum_{0\leq i<j+1}\ln K^{n_i}}.
\end{equation}
\end{lemma}

\begin{proof}
Write $X=\prod_{i\geq 0} A_i$ and $Y=\prod_{i\geq 0}\{0,\cdots, a_{i}-1\}$. If there are at most finitely many $a_i\geq2$, then (\ref{deqwow}) is obviously true. Then suppose there are infinitely many $a_i\geq2$. Note that, by Lemma \ref{deqouw}, $\dim_H Y=1$.

Write $A_i=\{\omega_{i,j}:0\leq j< a_{i}\}$. Let $\pi:X \rightarrow Y$, and suppose that
$$
\pi(\omega_{0,j_{0}},\omega_{1,j_{1}}\cdots)= j_{0}j_{1}\cdots.
$$
It is obvious that $\pi$ is a bijection.

For $x$, $y\in X$ with
$$
\sum_{0\leq i<j}{n_i}\leq\delta(x,y)<\sum_{0\leq i<j+1}{n_i},
$$
we have $\delta({\pi(x)},{\pi(y)})=j$ and
$$
\frac{\ln\rho( {\pi(x),\pi(y)})}{\ln\rho (x,y)}=\frac{\sum_{0\leq i<j}\ln a_i}{\ln K^{\delta(x,y)}}\geq\frac{\sum_{0\leq i<j}\ln a_i}{\sum_{0\leq i<j+1}\ln K^{n_i}}.
$$
So
$$
\begin{aligned}
\lim_{\delta\rightarrow 0^+}\inf&\left\{\frac{\ln\rho( {\pi(x),\pi(y)})}{\ln\rho (x,y)}:(x,y)\in X\times X,\ 0<\rho(x,y)<\delta\right\}\\
\geq&\liminf_{j\leftarrow \infty}\frac{\sum_{0\leq i<j}\ln a_i}{\sum_{0\leq i<j+1}\ln K^{n_i}}.
\end{aligned}
$$
Applying Lemma \ref{detouu} to this last inequality and using the fact that $\dim_H Y=1$, we get (\ref{deqwow}).
\end{proof}

Let $(X,f),(Y,{g})$ be TDSs and $\pi:X\to Y$ be surjective and continuous. If $\pi$ satisfies $\pi\circ f=g\circ\pi$, then we call $\pi$ a {\bf semi-conjugation} from $f$ to $g$. In this case we say $g$ is a {\bf factor} of $f$ and $f$ an {\bf extension} of $g$. If in addition $\pi$ is injective, then we call $\pi$ a {\bf conjugation} from $f$ to $g$.

Define
$$
\tau_K:W_K\rightarrow\mathbb{N},\ \tau_K(\omega)=\sum_{0\leq i<{|\omega|}}K^i\omega_i.
$$
For $n\geq1$, define
$$
\tau_{K,n}:\Sigma_K\rightarrow\Sigma_{K^n},\ (\tau_{K,n}(x))_i=\tau_K(x|_{\{in,\cdots,(i+1)n-1\}}),\ i\geq0.
$$

\begin{lemma}\label{detosx}
\rnc{\labelenumi}{(\alph{enumi})}
\begin{enumerate}
Let $n\geq1$.
\item $\tau_{K,n}$ is a conjugation from $\sigma_K^n$ to $\sigma_{K^n}$.
\item For $x,y\in\Sigma_K$,
$$
\rho(x,y)\leq\rho(\tau_{K,n}(x),\tau_{K,n}(y))\leq K^n\rho(x,y).
$$
\item For $X\subset\Sigma_K$ and $0\leq s\leq+\infty$,
$$
\mathscr{H}^{s}(X)\leq\mathscr{H}^{s}(\tau_{K,n}(X))\leq K^{sn}\mathscr{H}^{s}(X).
$$
So
$$
\dim_H\tau_{K,n}(X)=\dim_H X.
$$
\end{enumerate}
\end{lemma}

\begin{proof}
We prove (b) first. Let $x,y\in\Sigma_K$. If $x=y$, then the inequalities in (b) are obvious. Suppose $x\neq y$, $\delta(x,y)=in+j$, and $ 0\leq j<n$. Then $\delta(\tau_{K,n}(x),\tau_{K,n}(y))=i$. So
$$
\rho(x,y)=K^{-in-j}\leq K^{-in}=\rho(\tau_{K,n}(x),\tau_{K,n}(y))\leq K^{-in-j+n}=K^n\rho(x,y).
$$

Next we prove (a). By (b), $\tau_{K,n}$ is injective and continuous. Suppose $x\in\Sigma_{K^n}$. For $i\geq0$, $x_i\in\{0,\cdots,K^n-1\}$, and so there are unique $z_{i,j}\in \{0,\cdots,K-1\}$, $0\leq j< n$, with $x_i=\sum_{0\leq j< n}K^{j}z_{ij}$. Define $y\in\Sigma_K$ by $y_{in+j}=z_{ij}$ for $i\geq 0$ and $0\leq j< n$. Then $\tau_{K,n}(y)=x$. So $\tau_{K,n}$ is surjective. Now, using (b) again we see that $\tau_{K,n}^{-1}$ is continuous. So $\tau_{K,n}$ is a homeomorphism from $\Sigma_K$ to $\Sigma_{K^n}$.

Let $x\in\Sigma_K$. Then
$$
\begin{aligned}
(\sigma_{K^n}(\tau_{K,n}(x)))_i&=\sum_{0\leq j< n}K^jx_{(i+1)_{n+j}}\\
&=\sum_{0\leq j< n}K^j(\sigma_K^n(x))_{in+j}\\
&=(\tau_{K,n}(\sigma_K^n(x)))_i,\ i\geq 0.
\end{aligned}
$$
So $\sigma_{K^n}\circ \tau_{K,n}=\tau_{K,n}\circ \sigma_K^n$.

Now (c) follows from (b) and Lemma \ref{detviw}.
\end{proof}

Define
$$
\pi_K:\Sigma_K\times\Sigma_K  \rightarrow\Sigma_{K^2},\ (\pi_K(x_0,x_1))_i=x_{0,i}+Kx_{1,i}\in\{0,\cdots,K^{2}-1\},\ i\geq 0,
$$
where $(x_0,x_1)\in\Sigma_K\times\Sigma_K $, $x_{0,i},x_{1,i}\in\{0,\cdots,K-1\}$ are the $(i+1)$th coordinate values of $x_0,x_1$ respectively.

\begin{lemma}\label{detosw}
\rnc{\labelenumi}{(\alph{enumi})}
\begin{enumerate}
\item $\pi_K$ is a conjugation from $\sigma_K\times\sigma_K$ to $\sigma_{K^2}$.
\item For $(x_0,x_1),(y_0,y_1)\in\Sigma_K\times\Sigma_K $,
$$
\rho(\pi_K(x_0,x_1),\pi_K(y_0,y_1))=(\rho((x_0,x_1),(y_0,y_1)))^2.
$$
\item For $X\subset\Sigma_K\times\Sigma_K$ and $0\leq s\leq+\infty$,
$$
\mathscr{H}^{s}(X)=\mathscr{H}^{\frac{s}{2}}(\pi_K(X)).
$$
So
$$
\dim_H\pi_K(X)=\frac{1}{2}\dim_H X.
$$
\end{enumerate}
\end{lemma}

\begin{proof}
We prove (b) first. Let $(x_0,x_1),(y_0,y_1)\in\Sigma_K\times\Sigma_K $. If $(x_0,x_1)=(y_0,y_1)$, then
$$
\rho(\pi_K(x_0,x_1),\pi_K(y_0,y_1))=0=(\rho((x_0,x_1),(y_0,y_1)))^2.
$$
Suppose $(x_0,x_1)\neq(y_0,y_1)$. Let $i$ be the least natural number satisfying
$$
x_{0,i}\neq y_{0,i}\ \text{or}\ x_{1,i}\neq y_{1,i}.
$$
Then $\delta({\pi_K(x_0,x_1)},{\pi_K(y_0,y_1)})=i$. So
$$
\rho(\pi_K(x_0,x_1),\pi_K(y_0,y_1))=(K^2)^{-i}=(K^{-i})^2=(\rho((x_0,x_1),(y_0,y_1)))^2.
$$

Next we prove (a). By (b), $\pi_K$ is injective and continuous. Suppose $x\in\Sigma_{K^2}$. For $i\geq0$, $x_i\in\{0,\cdots,K^{2}-1\}$, and so there are unique $z_{0,i},z_{1,i}\in\{0,\cdots,K-1\} $ with $x_i=z_{0,i}+Kz_{1,i}$. Put $y_0=(z_{0,i})_{i\geq 0}$, $y_1=(z_{1,i})_{i\geq 0}$. Then $(y_0,y_1)\in\Sigma_K\times\Sigma_K $ and $\pi_K(y_0,y_1)=x$. So $\pi_K$ is surjective. Now, use (b) again and we see $\pi_K^{-1}$ is continuous. So $\pi_K$ is a homeomorphism from $\Sigma_K\times\Sigma_K $ to $\Sigma_{K^2}$.

Let $(x_0,x_1)\in\Sigma_K\times\Sigma_K $. Then
$$
\begin{aligned}
(\sigma_{K^2}(\pi_K(x_0,x_1)))_i&=(\pi_K(x_0,x_1))_{i+1}\\
&=x_{0,i+1}+Kx_{1,i+1}\\
&=(\sigma_K(x_0))_i+K(\sigma_K(x_1))_i\\
&=(\pi_K(\sigma_K(x_0),\sigma_K(x_1)))_i,\ i\geq 0.
\end{aligned}
$$
So $\sigma_{K^2}\circ\pi_K=\pi_K\circ(\sigma_{k}\times \sigma_{k} )$.

Now (c) follows from (b) and Lemma \ref{detviw}.
\end{proof}

\section{Variational inequality}\label{desoos}
In this section, we will prove a variational inequality (Lemma 5.8) for calculating the Hausdoff dimensions and Hausdorff measures of the sets $E_{\sigma_K}([p,q])$ and $D_{\sigma_K}([p,q])$.

Let $m\geq2$. Denote
\begin{equation}\label{ivoevvvwiz}
Q_m=\left\{p=(p_{0},\cdots,p_{m-1})\in{[0,1]}^m:\sum_{0\leq i<m}p_i=1\right\}.
\end{equation}
Let $\mathcal{K}=(K_i)_{0\leq i<m}$ be a partition of $\{0,\cdots,K-1\}$. Define

\begin{equation} \label{deqwuw}
f_{\mathcal{K}}:Q_m\rightarrow[0,+\infty),\ f_\mathcal{K} (p)=\sum_{0\leq i<m} -p_i\ln\frac{p_i}{\#(K_i)},
\end{equation}
where $0\ln0=0$.

\begin{lemma}\label{detoui}
\rnc{\labelenumi}{(\alph{enumi})}
\begin{enumerate}
\item $f_\mathcal{K}$ is continuous and strictly concave.
\item $f_\mathcal{K}(Q_{m})\subset[0,\ln K]$ and $f_\mathcal{K}(p)=\ln K$ if and only if $p_i=\frac{\#(K_i)}{K}$ for $0\leq i<m$.
\item Let $p_i\in{[0,1]}$, $2\leq i<m$, be fixed and satisfy
$$
1-\sum_{2\leq i<m}p_i=a>0.
$$
For $c\in{[0,1]}$, define $q_c\in Q_m$ by
$$
(q_c)_i=
\begin{cases}
ca,&i=0,\\
(1-c)a,&i=1,\\
p_i,&2\leq i<m.
\end{cases}
$$
Define $H:{[0,1]}\to[0,\ln K]$ by
$$
H(c)=f_\mathcal{K} ({q_c})=-ca\ln\frac{ca}{\#(K_0)}- (1-c)a\ln\frac{(1-c)a}{\#(K_1)}+\sum_{2\leq i<m} -p_i\ln\frac{p_i}{\#(K_i)}.
$$
Then $H$ is continuous and strictly concave and takes its maximal value at
$$
c=\frac{\#(K_0)}{\#(K_0)+\#(K_1)}.
$$
\end{enumerate}
\end{lemma}

\begin{proof}
Define $\phi:[0,+\infty)\rightarrow\mathbb{R}$, $\phi (x)=x\ln x$. Then $\phi$ is continuous and, since $\phi''(x)=\frac{1}{x}>0$, strictly convex.

(a) The continuity of $f_\mathcal{K}$ can be shown by the uniform continuity of the functions
$$
f_{\mathcal{K},j}:{[0,1]}\rightarrow[0,+\infty),\ f_{\mathcal{K},j}(x)=-x\ln\frac{x}{\#(K_j)},\ 0\leq j<m .
$$

Let $p_i\in Q_m$, $0<c_i<1$, $0\leq i<n$, with $\sum_{0\leq i<n}c_i=1$ and, for some $i_0,i_1$, $p_{i_0}\neq p_{i_1}$. Recall that $p_{i,j}$ is used  to denote the $(j+1)$th coordinate of $p_i$. Then
$$
\begin{aligned}
&f_\mathcal{K}\left(\sum_{0\leq i<n} c_{i}p_{i}\right)-\sum_{0\leq i<n}c_{i}f_{\mathcal{K}}(p_{i})\\
=&\sum_{0\leq j<m}\left(-\sum_{0\leq i<n}c_{i}p_{i,j}\cdot\ln\frac{\sum_{0\leq i<n}c_{i}p_{i,j}}{\#(K_j)}\right)\\
 &-\sum_{0\leq i<n}c_{i}\sum_{0\leq j<m}\left(-p_{i,j}\ln\frac{p_{i,j}}{\#(K_j)}\right)\\
=&\sum_{0\leq j<m}\left(-\sum_{0\leq i<n}c_{i}p_{i,j}\cdot\ln{\sum_{0\leq i<n}c_{i}p_{i,j}}\right)\\
 &-\sum_{0\leq i<n}c_i\sum_{0\leq j<m}\left(-p_{i,j}\ln p_{i,j}\right)\\
=&\sum_{0\leq j<m}\left(\sum_{0\leq i<n}c_i\phi(p_{i,j})-\phi\left(\sum_{0\leq i<n}c_{i}p_{i,j}\right)\right)\\
>&0\quad \text{(by the strict convexity of $\phi$)}.
\end{aligned}
$$
So $f_{\mathcal{K}}$ is strictly concave.

(b) Let $\mathcal{K}_1=(\{i\})_{0\leq i<K}$. Then for $q\in Q_K$,
$$
\begin{aligned}
f_{\mathcal{K}_1}(q)&=\sum_{0\leq i<K}(-q_i\ln q_i)=-\sum_{0\leq i<K}\phi ({q_i})=-K\sum_{0\leq i<K}\frac{1}{K}\phi ({q_i})\\
&\leq-K\phi\left(\sum_{0\leq i<K}\frac{1}{K}q_i\right)=-K\phi\left(\frac{1}{K}\right)=\ln K
\end{aligned}
$$
and $f_{\mathcal{K}_1}(q)=\ln K$ if and only if each $q_i=\frac{1}{K}$. Define $\psi:Q_{m}\rightarrow Q_K$ by
$$
(\psi(p))_i=\frac{p_{j_i}}{\#(K_{j_i})},\ 0\leq i<K ,
$$
where $j_i$ is the unique number with $i\in K_{j_i}$. Suppose $p\in Q_m$. Then
$$
\begin{aligned}
f_{\mathcal{K}}(p)&=\sum_{0\leq j<m}\left(-p_j\ln\frac{p_j}{\#(K_j)}\right)=\sum_{0\leq j<m}\sum_{i\in K_j}\left(-\frac{p_j}{\#(K_j)}\ln\frac{p_j}{\#(K_j)}\right)\\
&=\sum_{0\leq i<K}(-(\psi(p))_i\ln(\psi (p))_i)=f_{\mathcal{K}_1}({\psi(p)}).
\end{aligned}
$$
Then $f_{\mathcal{K}}=f_{\mathcal{K}_1}\circ\psi$. So
$$
\sup f_{\mathcal{K}}(Q_m)\leq\sup f_{\mathcal{K}_1}(Q_K) =\ln K
$$
and
$$
\begin{aligned}
f_{\mathcal{K}}(p)=\ln K&\Leftrightarrow(\psi(p))_i=\frac{1}{K}\ \text{for}\ 0\leq i<K\\
&\Leftrightarrow p_j=\frac{\#(K_j)}{K}\ \text{for}\ 0\leq j<m.
\end{aligned}
$$

(c) We have
$$
H'(c)=a\left(\ln\frac{1-c}{c}-\ln\frac{\#(K_1)}{\#(K_0)}\right).
$$
So
$$
H'\left(\frac{\#( K_0)}{\#(K_0)+\#( K_1)}\right)=0.
$$
Now the conclusion follows from the fact that
$$
H''(c)=-\frac{a}{c}-\frac {a}{1-c}<0,\ c\in(0,1).
$$
\end{proof}

We define the function $g_{\mathcal{K}}:Q_m\rightarrow[0,+\infty)$ by

\begin{equation} \label{deqiie}
\ g_{\mathcal{K}}(p)=\frac{f_{\mathcal{K}}(p)}{\ln K}=\frac{\sum_{0\leq i<m}- p_i\ln\frac{p_i}{\#(K_i)}}{\ln K}.
\end{equation}
By Lemma \ref{detoui}, for $g_{\mathcal{K}}$, we have  the following properties.

\begin{lemma}\label{detvxw}
\rnc{\labelenumi}{(\alph{enumi})}
\begin{enumerate}
\item $g_{\mathcal{K}}$ is continuous and strictly concave.
\item $g_{\mathcal{K}}(Q_m)\subset [0,1]$ and $g_{\mathcal{K}}(p)=1$ if and only if $p_i=\frac{\#(K_i)}{K}$ for $0\leq i<m$.
\item Let $p_i\in{[0,1]}$, $2\leq i<m$, be fixed and satisfy
$$
1-\sum_{2\leq i<m}p_i=a>0.
$$
For $c\in{[0,1]}$, define $q_c\in Q_m$ by
$$
(q_c)_i=
\begin{cases}
ca,&i=0,\\
(1-c)a,&i=1,\\
p_i,&2\leq i<m.
\end{cases}
$$
Define $h:{[0,1]}\rightarrow[0,1]$ by
$$
h(c)=g_{\mathcal{K}}({q_c})=\frac{-ca\ln\frac{ca}{\#(K_0)}- (1-c)a\ln\frac{(1-c)a}{\#(K_1)}+\sum_{2\leq i<m} -p_i\ln\frac{p_i}{\#(K_i)}}{\ln K}.$$
Then $h$ is continuous and strictly concave and takes its maximal value at
$$
c=\frac{\#(K_0)}{\#(K_0)+\#(K_1)}.\ \Box
$$
\end{enumerate}
\end{lemma}

For $n\geq1$, write
$$
Q_{m,n}=\{p\in Q_m:np_i\in\mathbb{N}\ \text{for}\ 0\leq i<m\}.
$$

\begin{lemma}\label{detous}
Let $p\in Q_{m}$ and $n\geq1$. Then there is $q\in Q_{m,n}$ with
\begin{equation} \label{deqoux}
\rho(p,q)<\frac{m-1}{n}.
\end{equation}
\end{lemma}

\begin{proof}
For $0\leq i<m-1$, let $q_i$ be the number with $p_i-\frac{1}{n}<q_i\leq p_i$ and $nq_i\in\mathbb{N}$. Put $q_{m-1}=1-\sum_{0\leq i<m-1}q_i$. Put $q=(q_i)_{0\leq i<m}$. Then $q\in Q_{m,n}$ and
$$
\rho(p,q)=\sup_{0\leq i<m}|p_i-q_i|<\frac{m-1}{n}.
$$
\end{proof}

\begin{lemma}\label{detouz}
$\#(Q_{m,n})\leq(n+1)^m$.
\end{lemma}

\begin{proof}
Let
$$
A=\left\{\left(\frac{i_j}{n}\right)_{0\leq j<m}:0\leq i_j < n+1\right\}.
$$
Then $\#( A)=(n+1)^m$ and $Q_{m,n}\subset A$.
\end{proof}

For $\omega\in W_{K,n}$, define $\mu_{\mathcal{K},n}(\omega)\in Q_m$ by
$$
(\mu_{\mathcal{K},n}(\omega))_i=\frac{1}{n}\#(\{0\leq j<n:\omega_j\in K_i\}),\ 0\leq i<m.
$$
For $p\in Q_{m,n}$, write
$$
A_{\mathcal{K},p,n}=\{\omega\in W_{K,n}:\mu_{\mathcal{K},n}(\omega)=p\}
$$
and
$$
a_{\mathcal{K},p,n}=\#(A_{\mathcal{K},p,n}).
$$
For $P\subset Q_{m,n}$, write
$$
A_{\mathcal{K},P,n}=\{\omega\in W_{K,n}:\mu_{\mathcal{K},n}(\omega)\in P\}
$$
and
$$
a_{\mathcal{K},P,n}=\#(A_{\mathcal{K},P,n}).
$$

The following lemma is direct.
\begin{lemma}\label{detoio}
For $n \geq 1$ and $p\in Q_{m,n}$,
\begin{equation} \label{deqoiu}
a_{\mathcal{K},p,n}=\prod_{0\leq i<m}C_{n-\sum_{0\leq j<i}np_j}^{np_i}(\#(K_i))^{np_i}=\frac{n!\prod_{0\leq i<m}(\#(K_i))^{np_i}}{\prod_{0\leq i<m}(np_i)!}.\ \Box
\end{equation}
\end{lemma}

We will now review Stirling's Formula, which says

\begin{equation} \label{deqoiz}
\ln n!=n\ln n-n+\ln\sqrt{2\pi n}+\epsilon_n\ \text{where}\ \lim_{n\rightarrow \infty}\epsilon_n=0.
\end{equation}

\begin{lemma}\label{detoei}
\begin{equation} \label{deqouo}
\left|\frac{1}{n}\ln a_{\mathcal{K},p,n}-f_{\mathcal{K}}(p)\right|\leq\frac{1}{n}{(m+1)(\ln\sqrt{2\pi n}+E)},\ \text{for}\ p\in Q_{m,n},
\end{equation}
where $E=\sup_{n\geq1}|\epsilon_n|$ with $\epsilon_n$ as in (\ref{deqoiz}). Thus
\begin{equation} \label{deqoez}
\lim_{n\rightarrow \infty}\sup\left\{\left|\frac{1}{n}\ln a_{\mathcal{K},p,n}-f_{\mathcal{K}}(p)\right|:p\in Q_{m,n}\right\}=0.
\end{equation}
\end{lemma}

\begin{proof}
Applying (\ref{deqoiz}) to (\ref{deqoiu}), we get
$$
\begin{aligned}
&\ln a_{\mathcal{K},p,n}\\
=&\sum_{0\leq i<m}np_i\ln\#(K_i)+\left(n\ln n-\sum_{0\leq i<m}np_i\ln np_i\right)-\left(n-\sum_{0\leq i<m}np_i\right)\\
&+\left(\ln\sqrt{2\pi n}-\sum_{0\leq i<m}\ln\sqrt{2\pi np_i}\right)+\left(\epsilon_n-\sum_{0\leq i<m}\epsilon_{np_i}\right)\\
=&\sum_{0\leq i<m}np_i\ln\#(K_i)+\left(n\ln n-\sum_{0\leq i<m}np_i\ln np_i\right)\\
&+\left(\ln\sqrt{2\pi n}-\sum_{0\leq i<m}\ln\sqrt{2\pi np_i}\right)+\left(\epsilon_n-\sum_{0\leq i<m}\epsilon_{np_i}\right)\\
=&\sum_{0\leq i<m}np_i\ln\#(K_i)-\sum_{0\leq i<m}np_i\ln p_i\\
&+\left(\ln\sqrt{2\pi n}-\sum_{0\leq i<m}\ln\sqrt{2\pi np_i}\right)+\left(\epsilon_n-\sum_{0\leq i<m}\epsilon_{np_i}\right)\\
=&-\sum_{0\leq i<m}np_i\ln\frac{p_i}{\#(K_i)}+\left(\ln\sqrt{2\pi n}-\sum_{0\leq i<m}\ln\sqrt{2\pi np_i}\right)+\left(\epsilon_n-\sum_{0\leq i<m}\epsilon_{np_i}\right)\\
=&nf_{\mathcal{K}}(p)+\left(\ln\sqrt{2\pi n}-\sum_{0\leq i<m}\ln\sqrt{2\pi np_i}\right)+\left(\epsilon_n-\sum_{0\leq i<m}\epsilon_{np_i}\right),
\end{aligned}
$$
i.e.,
\begin{equation} \label{deqoeo}
\begin{aligned}
&\frac{1}{n}\ln a_{\mathcal{K},p,n} \\
=&f_{\mathcal{K}}(p)+\frac{1}{n}\left(\ln\sqrt{2\pi n}-\sum_{0\leq i<m}\ln\sqrt{2\pi np_i}\right)+\frac{1}{n}\left(\epsilon_n-\sum_{0\leq i<m}\epsilon_{np_i}\right).
\end{aligned}
\end{equation}
Since $\left|\ln\sqrt{2\pi n}-\sum_{0\leq i<m}\ln\sqrt{2\pi np_i}\right|\leq(m+1)\ln\sqrt{2\pi n}$, (\ref{deqoeo}) leads to (\ref{deqouo}). Now (\ref{deqoez}) follows from (\ref{deqouo}) and the fact $\frac{\ln\sqrt{2\pi n}}{n}\rightarrow 0$.
\end{proof}

 Let $x\in\Sigma_K$. For $n\geq1$, define $\mu_{\mathcal{K},n}(x)\in Q_m$ by
$$
(\mu_{\mathcal{K},n}(x))_i=\frac{1}{n}\#(\{0\leq j<n:x_j\in K_i\}),\ 0\leq i<m.
$$
Define
$$
\mu_\mathcal{K} (x)=\omega(\mu_{\mathcal{K},n}(x):n\geq1)\subset Q_m.
$$
We call $\mu_\mathcal{K} (x)$ the {\bf distributional density spectrum} of $x$ on $\mathcal{K}$.

\begin{lemma}\label{detosi}
Let $x\in\Sigma_K$. Then $\mu_\mathcal{K} (x)\in\mathcal{C}(Q_m)$.
\end{lemma}

\begin{proof}
Apply Lemma \ref{detvxv} to
$$
\rho(\mu_{\mathcal{K},n}(x),\mu_{\mathcal{K},n+1}(x))\leq\frac{1}{n+1}\rightarrow 0, \ n\rightarrow \infty.
$$
\end{proof}

\begin{lemma}\label{detoso}
Let $\emptyset\neq P\subset Q_m$. Then $\dim_H\{x\in \Sigma_K:\mu_\mathcal{K} (x)\cap P\neq\emptyset\}\leq\sup g_\mathcal{K} (P)$.
\end{lemma}

\begin{proof}
Suppose $\sup g_\mathcal{K}(P) =s$.

Let $\epsilon>0$. By the uniform continuity of $f_\mathcal{K}$, we may take $\delta>0$ such that

\begin{equation} \label{deqwux}
|f_\mathcal{K}(p)-f_\mathcal{K}(q)|<\epsilon\ln K\ \text{for}\ p,q\in Q_m\ \text{with}\ \rho(p,q)<\delta.
\end{equation}
Take $n_0$ such that

\begin{equation} \label{deqwuu}
\frac{(m+1)(\ln\sqrt{2\pi n}+E)}{n}<\epsilon\ln K\ \text{for}\ n\geq n_0,
\end{equation}
where $E=\sup_{n\geq1}|\epsilon_n|$ with $\epsilon_n$ as in (\ref{deqoiz}).

Let
$$
P_i=Q_{m,i+1}\cap B_\delta (P),\ i\geq0.
$$
If $i\geq\frac {m-1}{\delta}-1$, then $\frac{m-1}{i+1}\leq\delta$ and, by Lemma \ref{detous}, $P_i\neq\emptyset$ and $P\subset B_\delta (P_i)$.

Let $n_1\geq\sup(n_0,\frac{m-1}{\delta}-1)$. Suppose $i\geq n_1$ and $q\in P_i$. Then
\begin{equation} \label{deqwus}
\begin{aligned}
&\ln a_{\mathcal{K},q,i+1}\\
\leq&(i+1)f_\mathcal{K}(q)+(m+1)(\ln\sqrt{2\pi (i+1)}+E)\quad\text{(by (\ref{deqouo}))}\\
\leq&(i+1)(f_\mathcal{K}(p)+\epsilon\ln K)+(i+1)\epsilon\ln K\quad\text{(by (\ref{deqwux}) and (\ref{deqwuu}))}\\
=&(i+1)(s\ln K+\epsilon\ln K)+(i+1)\epsilon\ln K\\
=&(i+1)(s+2\epsilon)\ln K.
\end{aligned}
\end{equation}
By (\ref{deqwus}), we have

\begin{equation} \label{deqwuz}
a_{\mathcal{K},P_i,i+1}\leq(i+2)^me^{(i+1)(s+2\epsilon)\ln K}=(i+2)^mK^{(i+1)(s+2\epsilon)}\ \text{for}\ n\geq n_1.
\end{equation}

Now, if $n\geq n_1$, it follows from (\ref{deqwuz}) that
\begin{equation}\label{q11021119}
\begin{aligned}
\sum_{\omega\in A_{\mathcal{K},P_i,i+1}}|[\omega]|^{s+4\epsilon}
&=a_{\mathcal{K},P_i,i+1}K^{-(i+1)(s+4\epsilon)}\\
&\leq(i+2)^mK^{(i+1)(s+2\epsilon)}K^{-(i+1)(s+4\epsilon)}\\
&=(i+2)^mK^{-2(i+1)\epsilon}\\
&=(i+2)^mK^{-(i+1)\epsilon}\cdot K^{-(i+1)\epsilon}.
\end{aligned}
\end{equation}
Note that $(i+2)^mK^{-(i+1)\epsilon}\rightarrow 0$. So, by (\ref{q11021119}), we can take $n_2\geq n_1$ such that
\nc{\deqwue}{(\ref{deqwue})} 
\begin{equation} \label{deqwue}
\sum_{\omega\in A_{\mathcal{K},P_i,i+1}}|[\omega]|^{s+4\epsilon}<K^{-(i+1)\epsilon}\ \text{for}\ i\geq n_2.
\end{equation}
By
$$
\sum_{i\geq n}K^{-(i+1)\epsilon}=\frac{K^{-(n+1)\epsilon}}{1-K^{-\epsilon}}\rightarrow 0,\ n\rightarrow \infty,
$$
we take $n_3\geq n_2$ such that

\begin{equation} \label{deqwun}
\sum_{i\geq n}K^{-(i+1)\epsilon}<1\ \text{for}\ n\geq n_3.
\end{equation}
Eqs (\ref{deqwue}) and (\ref{deqwun}) lead to

\begin{equation} \label{deqwso}
\sum_{i\geq n}\sum_{\omega\in A_{\mathcal{K},P_i,i+1}}|[\omega]|^{s+4\epsilon}<1\ \text{for}\ n\geq n_3.
\end{equation}

Suppose $x\in\{y\in \Sigma_K:\mu_\mathcal{K}(y)\cap P\neq\emptyset\}$. Then there are infinitely many $i$ with $\rho (\mu_{\mathcal{K},i+1}(x),P)<\delta$. For such $i$, we have $\mu_{\mathcal{K},i+1}(x)\in P_i$ and thus $x\in[A_{\mathcal{K},P_i,i+1}]$. So
$$
\{x\in \Sigma_K:\mu_\mathcal{K}(x)\cap P\neq\emptyset\}\subset\bigcup_{i\geq n}[A_{\mathcal{K},P_i,i+1}]\ \text{for}\ n\geq0,$$
and thus, for $n\geq0$, $\mathcal{A}_n:=\{[\omega]:\omega\in\ A_{\mathcal{K},P_i,i+1},\ i\geq n\}$ is a countable cover of $\{x\in \Sigma_K:\mu_\mathcal{K}(x)\cap P\neq\emptyset\}$.
Note that, for $[\omega]\in\mathcal{A}_n$, $|[\omega]|\leq K^{-(n+1)}$.

Suppose $n\geq n_3$. Now we have
$$
\begin{aligned}
\mathscr{H}_{K^{-n}}^{s+4\epsilon}(\{x\in \Sigma_K:\mu_\mathcal{K} (x)\cap P\neq\emptyset\})
&\leq \sum_{[\omega]\in\mathcal{A}_n}|[\omega]|^{s+4\epsilon}\\
&=\sum_{i\geq n}\sum_{\omega\in A_{\mathcal{K},P_i,i+1}}|[\omega]|^{s+4\epsilon}\\
&<1.
\end{aligned}
$$
Letting $n\rightarrow \infty$, we have $$ \mathscr{H}^{s+4\epsilon}(\{x\in \Sigma_K:\mu_\mathcal{K}(x)\cap P\neq\emptyset\})\leq 1.$$ Then $$\dim_H\{x\in \Sigma_K:\mu_\mathcal{K}(x) \cap P\neq\emptyset\}\leq s+4\epsilon.$$ Since $\epsilon>0$ was arbitrary, we have $$\dim_H\{x\in \Sigma_K:\mu_\mathcal{K}(x)\cap P\neq\emptyset\}\leq s.$$
\end{proof}

\section{Hausdorff dimensions of $E_{\sigma_K}(\mathcal{J})$ and $D_{\sigma_K}(\mathcal{J})$}\label{desooz}

We will recall some of the notation defined at the beginning of this section.
$$
\tau_K:W_K\rightarrow\mathbb{N},\ \tau_K(\omega)=\sum_{0\leq i<|\omega|}K^i\omega_i
$$
$$
\tau_{K,n}:\Sigma_K\rightarrow\Sigma_{K^n},\ (\tau_{K,n}(x))_i=\tau_K(x|_{\{in,\cdots,(i+1)n-1\}}),\ i\geq0,\ n\geq1
$$
$$
\pi_K:\Sigma_K\times\Sigma_K  \rightarrow\Sigma_{K^2},\ (\pi_K(x_0,x_1))_i=x_{0,i}+Kx_{1,i}\in \{0,\cdots,K^2-1\},\ i \geq 0
$$
$$
Q_2=\left\{p=(p_0,p_1)\in{[0,1]}^2:p_0+p_1=1\right\}
$$
Additionally, we write
$$E_K =\{i+Ki:0\leq i<K\}\subset\{0,\cdots,K^2-1\}.$$
Then
$$E_{K}^{n}=\prod_{0\leq i<n}E_K\subset\{0,\cdots,K^2-1\}^n,\ n\geq1.$$

\begin{lemma}\label{detonx}
Let $\emptyset\neq {\mathcal{J}}\subset{\mathcal{C}({[0,1]})}$. Then
$$
\dim_H E_{\sigma_K}({\mathcal{J}})\leq2-\inf\{\sup I:I\in{\mathcal{J}}\}.
$$
\end{lemma}

\begin{proof}
Let $X=\pi_K (E_{{\mathcal{J}}}(\sigma_K))$. Let

\begin{equation} \label{deqwoe}
q=\inf\{\sup I:I\in{\mathcal{J}}\}.
\end{equation}
By Lemma \ref{detosw}, to prove Lemma \ref{detonx}, it is enough to show $\dim_H X\leq1-\frac{q}{2}$. If $q=0$, the inequality is obvious. So we suppose $q>0$.

Let $P_{q}=\{r=(r_0,r_1)\in Q_2:r_0\geq {q}\}$. For $n \geq 1$, write
$$
\begin{aligned}
&K_{n,0}=\tau_{K^2}(E_{K}^{n}),\ K_{n,1}=\{0,\cdots,K^{2n}-1\}\setminus K_{n,0};\ \mathcal{K}_n=(K_{n,0},K_{n,1}).
\end{aligned}
$$

Suppose $x=\pi_K(y,z)\in X$, where $(y,z)\in E_{\mathcal{J}}({\sigma_K})=\mathscr{E}_{\sigma_K}^{-1}(\mathcal{J})$. Then for any $n\geq 1$,
$$
\begin{aligned}
q&\leq\sup \mathscr{E}_{\sigma_K}(y,z)\quad\text{(by $\mathscr{E}_{\sigma_K}(y,z)\in\mathcal{J}$ and (\ref{deqwoe}))}\\
&=\sup \mathscr{E}_{\sigma_{K}^{n}}(y,z)\quad\text{(by Lemma \ref{detoss})}\\
&=\lim_{i\rightarrow \infty}\limsup_{j\rightarrow \infty}\frac{1}{j}\#(\{0\leq k<j:y|_{\{kn,\cdots,(k+i)n-1\}} =z|_{\{kn,\cdots,(k+i)n-1\}}\})\\
&=\lim_{i\rightarrow \infty}\limsup_{j\rightarrow \infty}\frac{1}{j}\#(\{0\leq k<j:\sigma_{K^2}^{kn}(x)|_{\{0,\cdots,in-1\}}\in E_{K}^{in}\})\\
&=\lim_{i\rightarrow \infty}\limsup_{j\rightarrow \infty}\frac{1}{j}\#(\{0\leq k<j:\sigma_{K^{2n}}^{k}(\tau_{K^2,n}(x))|_{\{0,\cdots,i-1\}}\in K_{n,0}^{i}\})\\
&\leq\limsup_{j\rightarrow \infty}\frac{1}{j}\#(\{0\leq k<j:(\sigma_{K^{2n}}^{k}(\tau_{K^2,n}(x)))_0\in K_{n,0}\})\\
&=\sup\{p_0:p=(p_0,p_1)\in\mu_{\mathcal{K}_n}(\tau_{K^2,n}(x))\}\\
&\in\{p_0:p=(p_0,p_1)\in\mu_{\mathcal{K}_n}(\tau_{K^2,n}(x))\}\quad\text{(by the compactness of $\mu_{\mathcal{K}_n}(\tau_{K^2,n}(x))$)}.
\end{aligned}
$$
So $\mu_{\mathcal{K}_n}(\tau_{K^2,n}(x))\cap P_q\neq\emptyset$, i.e., $\tau_{K^2,n}(x)\in\{x'\in \Sigma_K:\mu_\mathcal{K} ({x'})\cap P_q\neq\emptyset\}$. Then

\begin{equation} \label{deqonw}
\tau_{K^2,n}(X)\subset\{x'\in \Sigma_K:\mu_\mathcal{K} ({x'})\cap P_q\neq\emptyset\}.
\end{equation}

Now

\begin{equation} \label{deqonx}
\begin{aligned}
\dim_H X&=\dim_H\tau_{K^2,n}(X)\quad\text{(by Lemma \ref{detosx})}\\
&\leq\dim_H\{x'\in \Sigma_K:\mu_\mathcal{K} ({x'})\cap P_q\neq\emptyset\}\quad\text{(by (\ref{deqonw}))}\\
&\leq\sup g_{\mathcal{K}_n}({P_q})\quad\text{(by Lemma \ref{detoso})}.\\
\end{aligned}
\end{equation}

Let $n\geq\frac{-\ln q}{\ln K}$. Then
\begin{equation}
\frac{\#( K_{n,0})}{\#(K_{n,0})+\#(K_{n,1})}=\frac{K^n}{K^{2n}}=\frac{1}{K^n}\leq q.
\label{1210290941}
\end{equation}
For $n\geq\frac{-\ln q}{\ln K}$, we have
\begin{equation} \label{deqonu}
\begin{aligned}
&\sup g_{\mathcal{K}_n}({P_q})\\
=&\sup g_{\mathcal{K}_n}\left(\left\{p=(p_0,p_1)\in Q_2:p_0\geq q\geq\frac{\#(K_{n,0})}{\#(K_{n,0})+\#(K_{n,1}}\right\}\right) \quad \text{by (\ref{1210290941})}\\
=&g_{\mathcal{K}_n}(q)\quad \text{(by Lemma \ref{detvxw} (c))}\\
=&\frac{-q\ln\frac q{K^n}-(1-q)\ln\frac {1-q}{K^{2n}-K^n}}{\ln K^{2n}}\\
=&\frac{-q\ln q-(1-q)\ln(1-q)}{\ln K^{2n}}\\
&+\frac{q\ln{K^n}+(1-q)\ln(K^{2n}-K^n)}{\ln K^{2n}}\\
\rightarrow&\frac12q+(1-q)\quad (n\rightarrow\infty)\\
=&1-\frac {q}{2}.
\end{aligned}
\end{equation}

In (\ref{deqonx}) let $n\rightarrow\infty$. Then using (\ref{deqonu}), we get $\dim_H X\leq1-\frac {q}{2}$.
\end{proof}

If $(X,\rho,f)$ is a TDS, then $(x,y)\in X\times X$ is said to be an {\bf asymptotical pair} if
$$
\lim_{i\rightarrow \infty}\rho(f^i(x),f^i(y))=0,
$$
a {\bf proximal pair} if
$$
\liminf_{i\rightarrow \infty}\rho(f^i(x),f^i(y))=0,
$$
a ($\delta$-){\bf distal pair} if
$$
\liminf_{i\rightarrow \infty}\rho(f^i(x),f^i(y))(\geq\delta)>0,
$$
and a ($\delta$-){\bf Li--Yorke pair} if
$$
\liminf_{i\rightarrow \infty}\rho(f^i(x),f^i(y))=0,\ \limsup_{i\rightarrow \infty}\rho(f^i(x),f^i(y))(\geq\delta)>0.
$$
We use Asym$(f)$, Prox$(f)$, Dist$(f)$ and LY$(f)$ to denote the set of asymptotical pairs, the set of proximal pairs,  the set of distal pairs and the set of Li--Yorke pairs of $f$, respectively.

The properties stated in the two lemmas below are direct.

\begin{lemma}\label{detiiu}
Let $(X,\rho,f)$ be a TDS. Then
$$
\begin{aligned}
&\mathrm{Asym} (f)\subset D_f({[1,1]}),\\
&\mathrm{Dist} (f)\subset D_f({[0,0]}),\\
&\mathrm{Prox} (f)=X\times X\setminus\mathrm{Dist}(f),\\
&\mathrm{LY} (f)=\mathrm{Prox}(f)\setminus\mathrm{Asym}(f)=X\times X\setminus(\mathrm{Dist}(f)\cup\mathrm{Asym}(f)).
\end{aligned}
$$
\end{lemma}

Recall that $N\subset\mathbb{N}$ is said to be {\bf syndetic} if for some $k\geq1$ we have $N-\{0,\cdots,k-1\} \supset \mathbb{N}$.

\begin{lemma}\label{detvvz}
Let $x,y\in\Sigma_K$.
\rnc{\labelenumi}{(\alph{enumi})}
\begin{enumerate}
\item $(x,y)$ is an asymptotic pair for $\sigma_K$ if and only if $\{i\geq0:x_i\neq y_i\}$ is finite.
\item $(x,y)$ is a distal pair for $\sigma_K$ if and only if $\{i\geq0:x_i\neq y_i\}$ is syndetic.  $\Box$
\end{enumerate}
\end{lemma}

Let $N\subset\mathbb{N}$ with both $N$ and $N^c$ infinite. Write
$$
\begin{aligned}
&L_N=N\setminus(N+1)=\{l_{N,0}<l_{N,1}<\cdots\},\\
&R_N=(N+1)\setminus N=\{r_{N,0}<r_{N,1}<\cdots\}.
\end{aligned}
$$
Note that $l_{N,i}<r_{N,i}<l_{N,i+1}$ for $i\geq0$, $N=\mathbb{N}\cap\bigcup_{i\geq 0}[l_{N,i},r_{N,i})$ and

\begin{equation} \label{deqwou}
R_N\subset L_{N^c}\subset R_N\cup\{0\}.
\end{equation}

\begin{lemma}\label{detvoe}
Let $N\subset\mathbb{N}$ with both $N$ and $N^c$ infinite. Then for $n\geq1$,
$$
\{i\geq0:\{i,\cdots,i+n-1\}\subset N\}=N\setminus(R_N-\{0,\cdots,n-1\}).
$$
\end{lemma}

\begin{proof}
Let $n\geq1$. Then $\{i,\cdots,i+n-1\}\subset N$ if and only if $i\in N$, and $\{i,\cdots,i+n-1\}\cap R_N=\emptyset$ if and only if $i\in N$ and $i\not\in R_N-\{0,\cdots,n-1\}$.
\end{proof}

Recall that $\zeta_n (N)=\#(N\cap\{0,\cdots,n-1\})$.

For $i\geq 0$, write
$$
\begin{aligned}
&t_{N,2i}=l_{N,i},\ t_{N,2i+1}=r_{N,i};\\
&d_{N,i}=t_{N,i+1}-t_{N,i};\\
&t_{N,i,j}=t_{N,i}+j,\ 0\leq j< d_{N,i};\\
&e_{N£¬i}=\zeta_{t_{N,i+1}} (N)=\sum_{0\leq 2j<i+1}d_{N,2j},\  f_{N,i}=\zeta_{t_{N,i+1}}(N^c)=t_{N,i+1}-e_{N,i}.\\
\end{aligned}
$$

\begin{lemma}\label{detosn}
Let $N\subset\mathbb{N}$ with both $N$ and $N^c$ infinite. Then
\rnc{\labelenumi}{(\alph{enumi})}
$$
\mu(N)=\left[\liminf_{i\rightarrow\infty}\frac{e_{N,2i+1}}{t_{N,2i+2}},\limsup_{i\rightarrow\infty}\frac{e_{N,2i}}{t_{N,2i+1}}\right].
$$
\end{lemma}

\begin{proof}
It is obvious that
$$
\mu(N)\supset\left[\liminf_{i\rightarrow\infty}\frac{e_{N,2i+1}}{t_{N,2i+2}},\limsup_{i\rightarrow\infty}\frac{e_{N,2i}}{t_{N,2i+1}}\right]
$$
Note that, for $0\leq j<d_{N,2i+2}$,

\begin{equation} \label{dsqoux}
\begin{aligned}
\frac{e_{N,2i+1}}{t_{N,2i+2}}&\leq\frac{e_{N,2i+1}+j+1}{t_{N,2i+2}+j+1}=\frac{\zeta_{t_{N,2i+2,j}+1}(N)}{t_{N,2i+2,j}+1}\\
&\leq\frac{e_{N,2i+1}+d_{N,2i+2}}{t_{N,2i+2}+d_{N,2i+2}}=\frac{e_{N,2i+2}}{t_{N,2i+3}}
\end{aligned}
\end{equation}
and for $0\leq j<d_{N,2i+1}$,

\begin{equation} \label{dsqouu}
\begin{aligned}
\frac{e_{N,2i+1}}{t_{N,2i+2}}&=\frac{e_{N,2i}}{t_{N,2i+2}}\leq\frac{e_{N,2i}}{t_{N,2i+1}+j+1}\\
&=\frac{\zeta_{t_{N,2i+1,j}+1}(N)}{t_{N,2i+1,j}+1}=\frac{\zeta_{t_{N,2i+1}}(N)}{t_{N,2i+1,j}+1}\\
&\leq\frac{\zeta_{t_{N,2i+1}}(N)}{t_{N,2i+1}}=\frac{e_{N,2i}}{t_{N,2i+1}}.
\end{aligned}
\end{equation}
Eqs (\ref{dsqoux}) and (\ref{dsqouu}) lead to
$$
\liminf_{i\rightarrow\infty}\frac{e_{N,2i+1}}{t_{N,2i+2}}\leq\inf\mu (N)\leq\sup\mu(N)\leq\limsup_{i\rightarrow\infty}\frac{e_{N,2i}}{t_{N,2i+1}}.
$$
So
$$
\mu(N)\subset\left[\liminf_{i\rightarrow\infty}\frac{e_{N,2i+1}}{t_{N,2i+2}},\limsup_{i\rightarrow\infty}\frac{e_{N,2i}}{t_{N,2i+1}}\right].
$$
\end{proof}

Let $[p,q]\in\mathcal{C}({[0,1]})$. Define
$$
\mathcal{M}({[p,q]})=\left\{N\subset\mathbb{N}:\frac{e_{N,2i+1}}{t_{N,2i+2}}\rightarrow p,\ \frac{e_{N,2i}}{t_{N,2i+1}}\rightarrow q\ \text{and}\ d_{N,i}\rightarrow\infty\right\}.
$$
Note that, by Lemma \ref{detosn}, for $N\in\mathcal{M}({[p,q]})$, $\mu(N)=[p,q]$.

\begin{lemma}\label{detvou}
(\cite{WS}) Let $[p,q]\in\mathcal{C}({[0,1]})$. Then $\mathcal{M}({[p,q]})\neq\emptyset$.
\end{lemma}

Let $N\subset\mathbb{N}$ with both $N$ and $N^c$ infinite. Define
$$
\gamma_N:N\rightarrow\mathbb{N},\ \gamma_N(n)=\zeta_n(N)=\#(N\cap\{0,\cdots,n-1\}).
$$
In fact, $\gamma_N$ is the unique order preserving bijective map from $N$ to $\mathbb{N}$. More intuitively, if $N=\{n_{0}<n_{1}<\cdots\}$, then $\gamma_N(n_i)=i$, $i\geq0$. Define
$\Phi_{N,a}:\Sigma_K\rightarrow\Sigma_K$ by
$$
(\Phi_{N,a}(x))_i=
\begin{cases}
a_{\gamma N (i)},&i\in N,\\
x_{\gamma{N^c}(i)},&i\in N^c.
\end{cases}
$$
More intuitively, in the case $l_{N,0}=0$,
$$
\begin{aligned}
\Phi_{N,a}(x)=&a_0\cdots a_{r_{N,0}-1}x_0\cdots x_{(l_{N,1}-r_{N,0})-1}\\
&a_{r_{N,0}}\cdots a_{r_{N,0}+(r_{N,1}-l_{N,1})-1}x_{l_{N,1}-r_{N,0}}\cdots x_{(l_{N,1}-r_{N,0})+(l_{N,2}-r_{N,1})-1}\\
&\cdots\\
=&a_0\cdots a_{d_{N,0}-1}x_{0}\cdots x_{d_{N,1}-1}\\
&a_{d_{N,0}}\cdots a_{d_{N,0}+d_{N,2}-1}x_{d_{N,1}}\cdots x_{d_{N,1}+d_{N,3}-1}\\
&\cdots\\
=&a_0\cdots a_{e_{N,0}-1}x_{f_{N,0}}\cdots x_{f_{N,1}-1}\\
&a_{e_{N,1}}\cdots a_{e_{N,2}-1}x_{f_{N,2}}\cdots x_{f_{N,3}-1}\\
&\cdots.
\end{aligned}
$$
Note that $\Phi_{N,a}$ is a continuous injection. The idea of the definition of $\Phi_{N,a}$ comes from \cite{Mis}. Our definition is slightly different from the corresponding one in \cite{Mis}.

\begin{lemma}\label{dstowv}
Let $a,b,c\in\Sigma_K$. Suppose $(a,b)$ is $(K^{-k+1})$-distal for $\sigma_K$.
Let $[p,q]\in\mathcal{C}({[0,1]})$ and $N\in\mathcal{M}({[p,q]})$. Write $x=\Phi_{N,c}(a)$, $y=\Phi_{N,c}(b)$. Then
$$
\mathcal{F}_{\sigma_K}((x,y),\epsilon)\equiv [{p,q}]\ \text{for}\ 0<\epsilon\leq K^{-k+1}.
$$
Thus, $(x,y)\in D_{[p,q]}({\sigma_K})$.
\end{lemma}

\begin{proof}
Suppose $0<\epsilon\leq K^{-k+1}$. Let $M=N_{\sigma_ K\times\sigma_K}((x,y),\Delta_\epsilon))$. Take $m$ with $K^{-m}<\epsilon$. Let
$$
M_m=(R_N-\{0,\cdots,m-1\})\cap\mathbb{N},\ M_k=(R_{N^c}-\{0,\cdots,k-1\})\cap\mathbb{N}.
$$
Since $\lim_{i\rightarrow \infty }t_{N,i+1}-t_{N,i}=+\infty$, Lemma \ref{dstoxm} implies $\mu (M_m)=\mu (M_k)=0$.

Suppose $i\in N\setminus M_m$. Then Lemma \ref{detvoe} implies that $\{i,\cdots,i+m-1\}\subset N$.
So, by the definition of $x,y$,
$$
\sigma_K^i(x)|_{\{0,\cdots,m-1\}}=\sigma_K^i(y)|_{\{0,\cdots,m-1\}}=\sigma_K^{\gamma_N(i)}(c)|_{\{0,\cdots,m-1\}},
$$
which means $\rho(\sigma_K^i(x),\sigma_K^i(y))\leq K^{-m}<\epsilon$ and thus $i\in M$. Then $N\setminus M_m\subset M$. So

\begin{equation} \label{deqwos}
\mu(M)\succeq\mu(N\setminus M_m)=\mu(N)=[p,q].
\end{equation}

Since $(a,b)$ is $(K^{-k+1})$-distal for $\sigma_K$, we may choose $t\in\mathbb{N}$ such that

\begin{equation} \label{deqoex}
\sigma_K^i(a)|_{\{0,\cdots,k-1\}}\neq\sigma_K^i(b)|_{\{0,\cdots,k-1\}}\ \text{for}\ i\geq t.
\end{equation}
Let $s$ be a number such that $\gamma_{N^c}(s)=t$. Suppose $i\in N^c\setminus M_k\setminus\{0,\cdots,s-1\}$. Then

\begin{equation} \label{deqoes}
\gamma_{N^c}(i)\geq t
\end{equation}
By Lemma \ref{detvoe},

\begin{equation} \label{deqoeu}
\{i,\cdots,i+k-1\}\subset N^c.
\end{equation}
Now (\ref{deqoes}), (\ref{deqoeu}), (\ref{deqoex}), and the definitions of $x$ and $y$ imply
$$
\sigma_K^i(x)|_{\{0,\cdots,k-1\}}=\sigma_K^{\gamma_{N^c}(i)}(a)|_{\{0,\cdots,k-1\}}\neq\sigma_K^{\gamma_{N^c}(i)}(b)|_{\{0,\cdots,k-1\}}=\sigma_K^i(y)|_{\{0,\cdots,k-1\}},
$$
which means $\rho(\sigma_K^i(x),\sigma_K^i(y))\geq K^{-k+1}\geq\epsilon$ and thus $i\in M^c$. Then
$$
N^c\setminus M_k\setminus\{0,\cdots,s-1\}\subset M^c,
$$
i.e. $M\subset N\cup M_k\cup\{0,\cdots,s-1\}$. Now

\begin{equation} \label{deqwoz}
\mu(M)\preceq\mu(N\cup M_k\cup\{0,\cdots,s-1\})=\mu(N)=[p,q].
\end{equation}
It follows from (\ref{deqwos}) and (\ref{deqwoz}) that $\mathcal{F}_{\sigma_K}((x,y),\epsilon)={[p,q]}$.
\end{proof}

\begin{lemma}\label{detosz}
Let ${[p,q]}\in\mathcal{C}({[0,1]})$. Then $\dim_H D_{\sigma_K}({[p,q]}) \geq2-{q}$.
\end{lemma}

\begin{proof}
By Lemma \ref{detvou}, pick $N\in\mathcal{M}({[p,q]})$. For $n\geq 1$, we define
$$
X_n=\prod_{i\geq 0}C_i,
$$
where
$$
C_i=
\begin{cases}
E_{K}^n,&i\in N,\\
W_{K^2,n} \setminus E_{K}^n ,&i\in N^c.
\end{cases}
$$
Let $Y_n=\pi_K^{-1} ({X_n})$.

Suppose $x=\pi_K(y,z)\in X_n$, where $(y,z)\in Y_n\subset\Sigma_K\times\Sigma_K $. Then there are $a,b,c\in\Sigma_K$ with
\begin{equation} \label{deqosw}
y=\Phi_{nN+\{0,\cdots,n-1\},a}(b),\ z=\Phi_{nN+\{0,\cdots,n-1\},a}(c)
\end{equation}
and

\begin{equation} \label{deqosx}
\pi_K(b,c)\in\prod_{i\geq 0}C_i.
\end{equation}
By (\ref{deqosx}) and the definitions of $B_i$ and $\pi_K$, for each $i\geq 0$, $b|_{\{i,\cdots,i+{2n-2}\}}\neq c|_{\{i,\cdots,i+{2n-2}\}}$, i.e., $\rho(\sigma_K^i(b),\sigma_K^i(c))\geq K^{-2n+2}$. So $(b,c)$ is a  $(K^{-2n+2})$-distal pair for $\sigma_K$. Note that, since $N\in\mathcal{M}({[p,q]})$, part (c) of Lemma \ref{dstoxm} implies that  $nN+\{0,\cdots,n-1\}\in\mathcal{M}({[p,q]})$. Then, by (\ref{deqosx}), (\ref{deqosw}) and Lemma \ref{dstowv}, $(y,z)\in D_{[p,q]}({\sigma_K})$. So

\begin{equation} \label{deqosu}
Y_n\subset D_{\sigma_K}({[p,q]}).
\end{equation}

Note that since, for each $k\geq 0$,
$$
K^n=\#(\{i+Ki:0\leq i<K\}^{n})\leq\#(C_k)\leq\#( W_{K^2,n} \backslash \{i+Ki:0\leq i<K\}^{n})=K^{2n}-K^n,
$$
we have
$$
\frac{\ln K^n}{\ln K^{2n}}\leq\frac{\sum_{0 \leq k < t_{N,i,j} }\ln\#( C_k)}{\sum_{0 \leq k < t_{N,i,j} }\ln K^{2n}}\leq\frac{\ln (K^{2n}-K^n)}{\ln K^{2n}}.
$$
Then

\begin{equation} \label{deqoss}
\begin{aligned}
&\frac{\sum_{0 \leq k < t_{N,2i,j}}\ln\#(C_k)}{\sum_{0 \leq k < t_{N,2i,j}+1}\ln K^{2n}}\\
&=\frac{\sum_{0 \leq k < t_{N,2i}}\ln\#( C_k)+\sum_{0 \leq k < j}\ln\#( C_{t_{N,2i}+k})}{( t_{N,2i}+j)\ln K^{2n}}
 \cdot\frac{t_{N,2i}+j}{t_{N,2i}+j+1}\\
&=\frac{e_{N,2i-1}\ln K^n+f_{N,2i-1}\ln (K^{2n}-K^n)+j\ln K^n}{(t_{N,2i}+j)\ln K^{2n}}
 \cdot\frac{t_{N,2i}+j}{t_{N,2i}+j+1}\\
&\geq\frac{e_{N,2i-1}\ln K^n+f_{N,2i-1}\ln (K^{2n}-K^n)+d_{N,2i}\ln K^n}{(t_{N,2i}+d_{N,2i})\ln K^{2n}}
 \cdot \frac{t_{N,2i}}{t_{N,2i}+1}\\
&=\frac{e_{N,2i}\ln K^n+f_{N,2i}\ln (K^{2n}-K^n)}{t_{N,2i+1}\ln K^{2n}}\cdot\frac{t_{N,2i}}{t_{N,2i}+1}\\
&\rightarrow\frac{q}{2}+(1-{q})\cdot\frac{\ln(K^{2n}-K^n)}{\ln K^{2n}},\ i\rightarrow \infty,
\end{aligned}
\end{equation}
and

\begin{equation} \label{deqosz}
\begin{aligned}
&\frac{\sum_{0 \leq k <t_{N,2i+1,j}}\ln\#( C_k)}{\sum_{0 \leq k <t_{N,2i+1,j}+1}\ln K^{2n}}\\
&=\frac{\sum_{0 \leq k <t_{N,2i+1}}\ln\#( C_k)+\sum_{0 \leq k <j}\ln\#( C_{t_{N,2i+1}+k})}{(t_{N,2i+1}+j)\ln K^{2n}}
 \cdot\frac{t_{N,2i+1}+j}{t_{N,2i+1}+j+1}\\
&=\frac{e_{N,2i}\ln K^n+f_{N,2i}\ln (K^{2n}-K^n)+j\ln (K^{2n}-K^n)}{(t_{N,2i+1}+j)\ln K^{2n}}
 \cdot\frac{t_{N,2i+1}+j}{t_{N,2i+1}+j+1}\\
&\geq\frac{e_{N,2i}\ln K^n+f_{N,2i}\ln (K^{2n}-K^n)}{t_{N,2i+1}\ln K^{2n}}\cdot\frac{t_{N,2i+1}}{t_{N,2i+1}+1}\\
&\rightarrow\frac{q}{2}+(1-{q})\cdot\frac{\ln(K^{2n}-K^n)}{\ln K^{2n}},\ i\rightarrow \infty.
\end{aligned}
\end{equation}
Eqs (\ref{deqoss}) and (\ref{deqosz}) lead to

\begin{equation} \label{deqose}
\liminf\frac{\sum_{0\leq i<j}\ln\#( C_i)}{\sum_{0\leq i<j+1}\ln K^{2n}}\geq\frac {q}2+(1-{q})\cdot\frac{\ln(K^{2n}-K^n)}{\ln K^{2n}}.
\end{equation}
Applying Lemma \ref{detoue} to (\ref{deqose}), we get

\begin{equation} \label{deqosn}
\dim_H X_n\geq\frac{q}{2}+(1-{q})\cdot\frac{\ln(K^{2n}-K^n)}{\ln K^{2n}}.
\end{equation}
Now

\begin{equation} \label{deqozo}
\begin{aligned}
\dim_{H}D_{\sigma_K}({[p,q]})&\geq\dim_{H} Y_n \quad\text{(by (\ref{deqosu}))}\\
&= 2\dim_{H} X_n \quad\text{(by Lemma \ref{detosw})}\\
&\geq 2\left(\frac{q}{2}+(1-{q})\cdot\frac{\ln(K^{2n}-K^n)}{\ln K^{2n}}\right) \quad\text{(by (\ref{deqosn}))}\\
&\rightarrow 2-q,\ n\rightarrow \infty.
\end{aligned}
\end{equation}
Since (\ref{deqozo}) holds for any $n\geq 1$, we have $\dim_{H} D_{\sigma_K}({[p,q]})\geq2-{q}$.
\end{proof}

\begin{lemma} \label{p1202141124}
Let $\emptyset\neq {\mathcal{J}}\subset{\mathcal{C}({[0,1]})}$. Then
$$
\dim_{H} D_{\sigma_K}({\mathcal{J}})\geq2-\inf\{\sup I:I\in{\mathcal{J}}\}.
$$
\end{lemma}

\begin{proof}
Let $q=\inf\{\sup I:I\in{{\mathcal{J}}}\}$. Suppose $\epsilon>0$. Pick $[p_0,p_1]\in{{\mathcal{J}}}$ with $p_1<q+\epsilon$. Using Lemma \ref{detosz}, we get
\begin{equation} \label{q1202141131}
\begin{aligned}
\dim_{H} D_{\sigma_K}({\mathcal{J}})\geq \dim_{H} D_{\sigma_K}({[p_0,p_1]}) =2-p_1>2-q-\epsilon.
\end{aligned}
\end{equation}
Since (\ref{q1202141131}) holds for any $\epsilon>0$, it follows that $\dim_{H} D_{\sigma_K}({\mathcal{J}})\geq2-q$.
\end{proof}

\begin{theorem}\label{detose}
Let $\emptyset\neq {{\mathcal{J}}}\subset{\mathcal{C}({[0,1]})}$. Then
$$
\dim_{H} E_{\sigma_K}({\mathcal{J}})=\dim_{H} D_{\sigma_K}({\mathcal{J}})=2-\inf\{\sup I:I\in{{\mathcal{J}}}\}.
$$
\end{theorem}

\begin{proof}
Apply Lemma \ref{detonx} and Lemma \ref{p1202141124} to $ D_{\sigma_K}({\mathcal{J}})\subset E_{\sigma_K}({\mathcal{J}})$.
\end{proof}

\begin{corollary}\label{detonw}
Let $[p,q]\in\mathcal{C}({[0,1]})$. Then
$$
\dim_{H} E_ {\sigma_K}({[p,q]})=\dim_{H} D_{\sigma_K}({[p,q]}) =2-{q}.
$$
\end{corollary}

\begin{corollary}\label{detonv}
For ${\mathcal{J}}\subset{\mathcal{C}({[0,1]})}$,
$$
\dim_{H} E_{\sigma_K}({\mathcal{J}})=\sup_{I\in{\mathcal{J}}}\dim_{H} E_{\sigma_K}({I}),\ \dim_{H} D_{\sigma_K}({\mathcal{J}})=\sup_{I\in{\mathcal{J}}}\dim_{H} D_{\sigma_K}({I}).
$$
\end{corollary}

\begin{proof}
This follows directly from Theorem \ref{detose} and Corollary \ref{detonw}.
\end{proof}

\begin{corollary}\label{detiio}
For $\sigma_K$, the distributional chaos relation with respect to DC1 and the distributional chaos relation with respect to DC2 are of Hausdorff dimension $1$.          $\Box$
\end{corollary}

\begin{theorem}\label{detiiv}
$$
\dim_{H}\rm{Asym}(\sigma_K)=1
$$
and
$$
\dim_{H}\rm{Prox}(\sigma_K)=\dim_{H}\rm{Dist}(\sigma_K)=\dim_{H}\rm{LY}(\sigma_K)=2.
$$
Moreover,
$$
\begin{aligned}
&\mathscr{H}^1(\rm{Asym}(\sigma_K))=+\infty,\\
&\mathscr{H}^2(\rm{Prox}(\sigma_K))=\mathscr{H}^2(\rm{LY}(\sigma_K))=1,\\
&\mathscr{H}^2(\rm{Dist}(\sigma_K))=0.\\
\end{aligned}
$$
\end{theorem}

\begin{proof}
Since $\rm{Asym}(\sigma_K)\subset D_{\sigma_K}([1,1])$ and $\dim_{H}D_{\sigma_K}([1,1])=1$, we have $\dim_{H}\rm{Asym}(\sigma_K)\leq1$. Since $\Delta({\Sigma_K})\subset\rm{Asym}(\sigma_K)$ and $\dim_{H}\Delta({\Sigma_K})=\dim_{H}\Sigma_K=1$, we have $\dim_{H}\rm{Asym}(\sigma_K)\geq1$. Thus $\dim_{H}\rm{Asym}(\sigma_K)=1$. Next we show $\mathscr{H}^1(\rm{Asym}(\sigma_K))=+\infty$.

Note that the map
$$
\varepsilon_K:\Sigma_K\rightarrow E_{K}^{\mathbb{N}}=\prod_{i\geq0}E_K\subset\Sigma_{K^2},\ (\varepsilon_K(x))_i=x_i+Kx_i,\ i\geq0
$$
is a homeomorphism with
\begin{equation} \label{deqwio}
\rho(\varepsilon_K(x),\varepsilon_K(y))=(K^2)^{-\delta(x,y)}=(K^{-\delta(x,y)})^2=(\rho(x,y))^2,\ x,y\in\Sigma_K.
\end{equation}
Eq. (\ref{deqwio}) and Lemma \ref{detviw} lead to
\begin{equation} \label{deqwon}
\mathscr{H}^{\frac{1}{2}}(E_{K}^{\mathbb{N}})=\mathscr{H}^1(\Sigma_K)=1.
\end{equation}

For $n\geq0$ and $\omega\in W_{K^2,n}$, write
$$
\omega E_{K}^{\mathbb{N}}=\{x\in\Sigma_{K^2}:x|_{\{0,\cdots,n-1\}}=\omega,\ x|_{\{n,n+1,\cdots\}}\in E_{K}^{\mathbb{N}}\}
$$
Then the map
$$
\sigma_{K^2}^{n}:\omega E_{K}^{\mathbb{N}}\rightarrow E_{K}^{\mathbb{N}}
$$
is a homeomorphism with
\begin{equation} \label{deqwii}
\rho(\sigma_{K^2}^{n}(x),\sigma_{K^2}^{n}(y))=K^{2n}\rho(x,y),\ x,y\in \omega E_{K}^{\mathbb{N}}.
\end{equation}
Applying Lemma \ref{detviw} to (\ref{deqwii}) and then using (\ref{deqwon}), we get

\begin{equation} \label{deqwiv}
\mathscr{H}^{\frac{1}{2}}(\omega E_{K}^{\mathbb{N}})=K^{-n}\mathscr{H}^{\frac{1}{2}}(E_{K}^{\mathbb{N}})=K^{-n}.
\end{equation}
Let
$$
X_n=\bigcup_{\omega\in W_{K^2,n}}\omega E_{K}^{\mathbb{N}}.
$$
Since each $\omega E_{K}^{\mathbb{N}}$ is compact and thus $\mathscr{H}^{\frac{1}{2}}$ measurable (see, e.g., \cite{Fal}), by (\ref{deqwiv}),
\begin{equation} \label{deqwiw}
\mathscr{H}^{\frac{1}{2}}(X_n)=\#(W_{K^2,n})K^{-n}=K^{2n}K^{-n}=K^n.
\end{equation}
Let $X=\pi_K(\rm{Asym}(\sigma_K))$. Then

\begin{equation} \label{deqwix}
X=\bigcup_{n\geq 0}X_n.
\end{equation}
By (\ref{deqwix}) and (\ref{deqwiw}),
$$
\mathscr{H}^{\frac{1}{2}}(X)\geq K^n\ \text{for}\ n\geq0.
$$
So $\mathscr{H}^{\frac{1}{2}}(X)=+\infty$. Now Lemma \ref{detosw} implies that $\mathscr{H}^{1}(\rm{Asym}(\sigma_K))=+\infty$.

For $n\geq 1$, let $Y_n=\prod_{i\geq 0}(W_{K^2,n} \backslash \{i+Ki:0\leq i<K\}^{n})$. We have
$$
\dim_{H} Y_n=\frac{\ln\#(C_n)}{\ln K^{2n}}=\frac{\ln(K^{2n}-K^n)}{\ln K^{2n}}.
$$
Then
\begin{equation} \label{deqiiw}
\dim_{H} Y_n<1\ \text{and}\ \dim_{H} Y_n\rightarrow1,\ n\rightarrow\infty.
\end{equation}

Suppose $y,z\in\Sigma_K$ with $\pi_K(y,z)\in Y_n$. Because $y|_{\{in,\cdots,(i+1)n-1\}}\neq z|_{\{in,\cdots,(i+1)n-1\}}$ for $i\geq 0$, we have $y|_{\{i,\cdots,i+2n-2\}}\neq z|_{\{i,\cdots,i+2n-2\}}$ for $i\geq 0$. So $(y,z)$ is a $(K^{-2n+2})$-distal pair for $\sigma_K$. Then

\begin{equation} \label{deqivv}
Y_n\subset\pi_K(\rm{Dist}(\sigma_K)).
\end{equation}
On the other hand, suppose $(y,z)$ is a distal pair for $\sigma_K$. Then there is some $n\geq 0$ with $\inf_{i\geq 0}\rho(\sigma_K^i(y),\sigma_K^i(z))\geq K^{-n}$, i.e., $x|_{\{i,\cdots,i+n\}}\neq y|_{\{i,\cdots,i+n\}}$ for each $i\geq 0$, thus $\pi_K(y,z)\in Y_{n+1}$. So

\begin{equation} \label{deqivw}
\pi_K(\rm{Dist}(\sigma_K))\subset\bigcup_{n\geq 1}Y_n.
\end{equation}
Eqs (\ref{deqivv}) and (\ref{deqivw}) lead to

\begin{equation} \label{deqiix}
\pi_K(\rm{Dist}(\sigma_K))=\bigcup_{n\geq 1}Y_n.
\end{equation}
Applying Lemma \ref{detoex} to (\ref{deqiiw}), we obtain

\begin{equation} \label{deqiiu}
\dim_{H} \bigcup_{n\geq 1}Y_n=1\ \text{and}\ \mathscr{H}^1\left(\bigcup_{n\geq 1}Y_n\right)=0.
\end{equation}
Applying Lemma \ref{detosw} and (\ref{deqiiu}) to (\ref{deqiix}), we have
$$
\dim_{H} \rm{Dist}(\sigma_K)=2\ \text{and}\ \mathscr{H}^2(\rm{Dist}(\sigma_K))=0.
$$
Since $\mathscr{H}^2(\rm{Dist}(\sigma_K))=\mathscr{H}^2(\rm{Asym}(\sigma_K))=0$, by Lemma \ref{detoex}, we have
$$
\mathscr{H}^2(\rm{Prox}(\sigma_K))=\mathscr{H}^2(X\times X\setminus\rm{Dist}(\sigma_K))=\mathscr{H}^2(X\times X)=1
$$
and
$$
\mathscr{H}^2(\rm{LY}(\sigma_K))=\mathscr{H}^2(\rm{Prox}(\sigma_K)\setminus\rm{Asym}(\sigma_K))=\mathscr{H}^2(\rm{Prox}(\sigma_K))=1.
$$
So
$$
\dim_{H}\rm{LY}(\sigma_K)=\dim_{H}\rm{Prox}(\sigma_K)=2.
$$
\end{proof}

\section{The Hausdorff measures of $E_{\sigma_K}({[p,q]})$ and $D_{\sigma_K}({[p,q]})$} \label{desooe}

We will now review some measure theoretical properties of $\mathscr{H}^1$ for $(\Sigma_K,\sigma_K)$.
For each column $[\omega]$ in $\Sigma_K$, $\mathscr{H}^1([\omega])=K^{-|\omega|}$.
Let $\mathcal{B}_{\Sigma_K}$ be the set of Borel subsets of $\Sigma_K$.
Then each member of $\mathcal{B}_{\Sigma_K}$ is $\mathscr{H}^1$ measurable, so $\mathscr{H}^1$ is a probability measure on $(\Sigma_K,\mathcal{B}_{\Sigma_K})$.
The measure $\mathscr{H}^1$ is ergodic for $\sigma_K$.
We say $x\in\Sigma_K$ is a \textbf{generic point} for $(\mathscr{H}^1,\sigma_K)$ provided $\frac{1}{j}\sum_{0\leq i<j}\delta_{\sigma_K^i(x)}\to \mathscr{H}^1$, $j\to\infty$, under the weak* topology, where $\delta_x$ is the measure $\delta_x(A)=1\Leftrightarrow x\in A$.
Let $G_{\mathscr{H}^1,\sigma_K}$ denote the set of generic points for $(\mathscr{H}^1,\sigma_K)$.
Then $G_{\mathscr{H}^1,\sigma_K}\in\mathcal{B}_{\Sigma_K}$, $\mathscr{H}^1(G_{\mathscr{H}^1,\sigma_K})=1$ and
\begin{equation}\label{e:6.1}
\mu(N_{\sigma_K}(x,A))=\mathscr{H}^1(A)\text{ for }x\in G_{\mathscr{H}^1,\sigma_K}\text{ and }A\in\mathcal{B}_{\Sigma_K}.
\end{equation}
See, e.g., \cite{b:42}.

\begin{lemma}\label{l:6.1}
Suppose $(x,y)\in\Sigma_K\times\Sigma_K$ with $\pi_K(x,y)\in G_{\mathscr{H}^1,\sigma_{K^2}}$. Then
$$\mathcal{F}_{\sigma_K}((x,y),K^{-k+1})=K^{-k}\text{ for }k\geq 0.$$
In particular, $(x,y)\in E_{\sigma_K}([0,0])$.
\end{lemma}
\begin{proof}
Let $\pi_K(y,z)=x\in G_{\mathscr{H}^1,\sigma_{K^2}}$. For $k\geq 0$, by the definition of $\pi_K$ and (\ref{e:6.1}), we have
$$\begin{aligned}
\mathcal{F}_{\sigma_K}((y,z),K^{-k+1})&=\mu(\{i\ge 0:y|_{[i,i+k)}=z|_{[i,i+k)}\})\\
&=\mu(N_{\sigma_{K^2}}(x,[E_K^k]))=(K^2)^{-k}K^k=K^{-k}.
\end{aligned}$$
\end{proof}

\begin{lemma}\label{l:6.2}
$\mathscr{H}^2(E_{\sigma_K}([0,0]))=1$.
\end{lemma}
\begin{proof}
Lemma \ref{l:6.1} implies $\pi_K^{-1}(G_{\mathscr{H}^1,\sigma_{K^2}})\subset E_{[0,0]}(\sigma_K)$. So it follows from Lemma $\ref{detosw}$ that
$$1\geq \mathscr{H}^2(E_{[0,0]}(\sigma_K))\ge \mathscr{H}^2(\pi_K^{-1}(G_{\mathscr{H}^1,\sigma_{K^2}}))=\mathscr{H}^1(G_{\mathscr{H}^1,\sigma_K})=1.$$
\end{proof}

\begin{lemma}\label{l:6.3}
Suppose $[p,q]\in\mathcal{C}([0,1])$, $N\in\mathcal{M}([p,q])$, $a,b,c\in\Sigma_K$ with $\pi_K(b,c)=x\in G_{\mathscr{H}^1,\sigma_K}$ and $y=\Phi_{N,c}(a)$, $z=\Phi_{N,c}(b)$. Then
$$\mathcal{F}_{\sigma_K}((y,z),t)=
\left\{\begin{aligned}
&[p+(1-p)K^{-k},q+(1-q)K^{-k}],&&K^{-k}<t\leq K^{-k+1},\,\,k\geq 1\\
&1,&&t>1.
\end{aligned}
\right.
$$
In particular, $(y,z)\in E_{\sigma_K}([p,q])$.
\end{lemma}
\begin{proof}
If $t>1=diam(\Sigma_K)$, then $\mathcal{F}_{\sigma_K}((y,z),t)=1$. Let $k\geq 1$ be a fixed integer and $K^{-k}<t\leq K^{-k+1}$.

Let $$\begin{aligned}
N_1&=\{i\geq 0:\rho(\sigma_K^i(y),\sigma_K^i(z))<t\}=\{i\geq 0:\rho(\sigma_K^i(y),\sigma_K^i(z))<K^{-k}\},\\
N_2&=\{i\geq 0:\rho(\sigma_K^i(b),\sigma_K^i(c))<t\}=\{i\geq 0:\rho(\sigma_K^i(b),\sigma_K^i(c))<K^{-k}\}\\
&=\{i\geq 0:\sigma_{K^2}^i(x)\in [E_K^k]\},\\
N_3&=(\{t_{N,2j+1}:j\geq 0\}-\{0,1,\cdots,k-1\})\cap\mathbb{N},\\
N_4&=N\setminus N_3,\\
N_5&=(\{t_{N,2j}:j\geq 0\}-\{0,1,\cdots,k-1\})\cap\mathbb{N},\\
N_6&=N^c\setminus N_5,\\
N_7&=\gamma_{N^c}(N_5),\\
N_8&=\gamma_{N^c}(N_6)=N^c_7.\\
\end{aligned}$$
To prove the lemma, it is enough to show
\begin{equation}\label{e:6.2}
  \mu(N_1)=[p+(1-p)K^{-k},q+(1-q)K^{-k}].
\end{equation}

Since $x\in G_{\mathscr{H}^1,\sigma_{K^2}}$, we have
\begin{equation}\label{e:6.3}
  \mu(N_2)=\mathscr{H}^1([E_K^k])=(K^2)^{-k}K^k=K^{-k}.
\end{equation}
Since $t_{N,2j+3}-t_{N,2j+1}=d_{N,2j+1}+d_{N,2j+2}\rightarrow \infty$, $t_{N,2j+2}-t_{N,2j}=d_{N,2j}+d_{N,2j+1}\rightarrow \infty$,
$\gamma_{N^c}(t_{N,2j+2})-\gamma_{N^c}(t_{N,2j})=d_{N,2j+1}\rightarrow\infty$, then, by Lemma \ref{dstoxm},
\begin{equation}\label{e:6.4}
  \mu(N_3)=\mu(N_5)=\mu(N_7)=0.
\end{equation}
Then
\begin{equation}\label{e:6.5}
  \mu(N_4 \sqcup N_6)=\mu(N_8)=1
\end{equation}

Let
$$J=\omega\left(\frac{\zeta_n(N_1\cap(N_4\sqcup N_6))}{n}:n\in(N_4\sqcup N_6)+1\right).$$
By (\ref{e:6.5}) and Lemma \ref{dstoxm}, to prove (\ref{e:6.2}) it is enough to prove
\begin{equation}\label{e:6.6}
  \inf J=p+(1-p)K^{-k},\,\,\sup J=q+(q-1)K^{-k}.
\end{equation}

Suppose $i\in N_4$. Then, for some $j\geq 0$, $i\in[t_{N,2j},t_{N,2j+1}-k+1)\cap\mathbb{N}\subset N$.
Then $i+\{0,1,\cdots,k-1\}\subset N$ and $y|_{[i,i+k)}=z|_{[i,i+k)}=a|_{[i-f_{N,2j},i-f_{N,2j}+k)}$.
Thus $\rho(\sigma_K^i(y),\sigma_K^i(z))\le K^{-k}$, which means $i\in N_1$.
So
\begin{equation}\label{e:6.7}
  N_4\subset N_1.
\end{equation}

Suppose $i\in N_6$. Then, $i\in [0,t_{N,0}-k+1)\cap\mathbb{N}\subset N^c$ or, for some $j\geq 0$, $i\in [t_{N,2j+1},t_{N,2j+2}-k+1)\cap\mathbb{N}\subset N^c$.
Then $i+\{0,1,\cdots,k-1\}\subset N^c$ and $y|_{[i,i+k)}=b|_{[\gamma_{N^c}(i),\gamma_{N^c}(i)+k)}$, $z|_{[i,i+k)}=c|_{[\gamma_{N^c}(i),\gamma_{N^c}(i)+k)}$,
Thus
$$\rho(\sigma_K^i(y),\sigma_K^i(z))\leq K^{-k}\Leftrightarrow\rho(\sigma_K^{\gamma_{N^c}(i)}(b),\sigma_K^{\gamma_{N^c}(i)}(c))\leq K^{-k}.$$
So, if $i\in N_1$, then $\gamma_{N^c}(i)\in N_2$.
Hence
\begin{equation}\label{e:6.8}
  \gamma_{N^c}(N_1\cap N_6)\subset N_2\cap\gamma_{N^c}(N_6)=N_2\cap N_8.
\end{equation}
On the other hand, suppose $i\in N_8$. Then, $\gamma_{N^c}^{-1}(i)\in [0,t_{N,0}-k+1)\cap\mathbb{N}\subset N^c$ or, for some $j\ge 0$, $\gamma_{N^c}^{-1}(i)\in [t_{N,2j+1},t_{N,2j+2}-k+1)\cap\mathbb{N}\subset N^c$.
Then $\gamma_{N^c}^{-1}(i)+\{0,1,\cdots,k-1\}\subset N^c$ and $y|_{[\gamma_{N^c}^{-1}(i),\gamma_{N^c}^{-1}(i)+k)}=b|_{[i,i+k)}$, $z|_{[\gamma_{N^c}^{-1}(i),\gamma_{N^c}^{-1}(i)+k)}=c|_{[i,i+k)}$,
Thus
$$\rho(\sigma_K^{\gamma_{N^c}^{-1}(i)}(y),\sigma_K^{\gamma_{N^c}^{-1}(i)}(z))\leq K^{-k}\Leftrightarrow\rho(\sigma_K^{i}(b),\sigma_K^{i}(c))\leq K^{-k}.$$
So, if $i\in N_2$, then $\gamma_{N^c}^{-1}(i)\in N_1$.
Then
\begin{equation}\label{e:6.9}
  \gamma_{N^c}^{-1}(N_2\cap N_8)\subset N_1\cap\gamma_{N^c}^{-1}(N_8)=N_1\cap N_6.
\end{equation}
Eqs (\ref{e:6.8}) and (\ref{e:6.9}) lead to
\begin{equation}\label{e:6.10}
  N_1\cap N_6=\gamma_{N^c}^{-1}(N_2\cap N_8).
\end{equation}

Combining (\ref{e:6.7}) and (\ref{e:6.10}) we get
\begin{equation}\label{e:6.11}
  N_1\cap (N_4\sqcup N_6)=N_4 \sqcup (N_1\cap N_6)=N_4 \sqcup \gamma_{N^c}^{-1}(N_2\cap N_8).
\end{equation}

Suppose $p=q=0$, i.e., $\mu(N)=0$. For $n\in \gamma_{N^c}^{-1}(N_2\cap N_8)$ we have
$$\begin{aligned}
&\frac{\zeta_n(\gamma_{N^c}^{-1}(N_2\cap N_8)}{n}\\
=&\frac{\#(\gamma_{N^c}^{-1}(N_2\cap N_8)\cap\{0,1,\cdots,n-1\})}{n}\\
=&\frac{\#((N_2\cap N_8)\cap\{0,1,\cdots,\gamma_{N^c}(n)-1\})}{n}\text{ (since $\gamma_{N^c}(n)\in N_2\cap N_8$)}\\
=&\frac{\zeta_{\gamma_{N^c}(n)}(N_2\cap N_8)}{n}\\
=&\frac{\zeta_{\gamma_{N^c}(n)}(N_2\cap N_8)}{\gamma_{N^c}(n)}\cdot\frac{\gamma_{N^c}(n)}{n}\\
\rightarrow& K^{-k}\text{ while }n\rightarrow\infty\\
&\text{(by $\mu(N_2)=K^{-k}$, $\mu(N_8)=\mu(N^c)=1$ and Lemma \ref{dstoxm})}.
\end{aligned}$$
Applying Lemma \ref{dstoxm} for this limit we see $\mu(\gamma_{N^c}^{-1}(N_2\cap N_8))=K^{-k}$. Since $\mu(N_4)\leq\mu(N)=0$, we have
\begin{equation}\label{e:6.12}
  \mu(N_4\sqcup\gamma_{N^c}^{-1}(N_2\cap N_8))=\mu(\gamma_{N^c}^{-1}(N_2\cap N_8))=K^{-k}.
\end{equation}
Now (\ref{e:6.6}) follows from (\ref{e:6.12}) and (\ref{e:6.11}).

Suppose $q>0$.
As $d_{N,j}\rightarrow\infty$, we may choose $j^*\ge k$ such that if $j\ge j^*$, then $d_{N,j}\ge k$.

Suppose $j\ge j^*$. Then $t_{N,2j+1}-k\in N_4$ and $t_{N,2j+2}-k\in N_6$.
Now
$$\begin{aligned}
&\frac{\zeta_{t_{N,2j+1}-k+1}(N_4\sqcup(N_1\cap N_6))}{t_{N,2j+1}-k+1}\\
=&\frac{\zeta_{t_{N,2j+1}-k+1}(N_4)}{t_{N,2j+1}-k+1}
+\frac{\zeta_{t_{N,2j+1}-k+1}(N_1\cap N_6)}{t_{N,2j+1}-k+1}\\
=&\frac{\zeta_{t_{N,2j+1}-k+1}(N_4)}{t_{N,2j+1}-k+1}
+\frac{\zeta_{t_{N,2j}-k+1}(N_1\cap N_6)}{t_{N,2j+1}-k+1}
\text{ (since $[t_{N,2j}-k+1,t_{N,2j+1}-k+1)\cap N_6=\emptyset$)}\\
=&\frac{\zeta_{t_{N,2j+1}-k+1}(N_4)}{t_{N,2j+1}-k+1}
+\frac{\zeta_{t_{N,2j}-k+1}(\gamma_{N^c}^{-1}(N_2\cap N_8))}{t_{N,2j+1}-k+1}\text{ (by (\ref{e:6.10}))},
\end{aligned}$$
where
$$\begin{aligned}
&\frac{\zeta_{t_{N,2j+1}-k+1}(N_4)}{t_{N,2j+1}-k+1}\\
=&\left(\frac{\zeta_{t_{N,2j+1}}(N)}{t_{N,2j+1}}
-\frac{\zeta_{t_{N,2j+1}}(N_3)}{t_{N,2j+1}}\right)
\cdot\frac{t_{N,2j+1}}{t_{N,2j+1}-k+1}
\text{ (since $[t_{N,2j}-k+1,t_{N,2j+1})\cap\mathbb{N}\subset N_3$)}\\
\rightarrow&q\text{ while }i\rightarrow\infty\text{ (since $\frac{\zeta_{t_{N,2j+1}}(N)}{t_{N,2j+1}}\rightarrow q$ and $\mu(N_3)=0$)}
\end{aligned}$$
and
$$
\begin{aligned}
&\frac{\zeta_{t_{N,2j}-k+1}(\gamma_{N^c}^{-1}(N_2\cap N_8))}{t_{N,2j+1}-k+1}\\
=&\frac{\zeta_{\gamma_{N^c}(t_{N,2j}-k)+1}(N_2\cap N_8)}{t_{N,2j+1}-k+1}\,\,(\text{since } \gamma_{N^c}(t_{N,2j}-k)\in N_8)\\
=&\frac{\zeta_{\gamma_{N^c}(t_{N,2j}-k)+1}(N_2)}{t_{N,2j+1}-k+1}
-\frac{\zeta_{\gamma_{N^c}(t_{N,2j}-k)+1}(N_7)}{t_{N,2j+1}-k+1}\\
=&\left(\frac{\zeta_{\gamma_{N^c}(t_{N,2j}-k)+1}(N_2)}{\gamma_{N^c}(t_{N,2j}-k)+1}
-\frac{\zeta_{\gamma_{N^c}(t_{N,2j}-k)+1}(N_7)}{\gamma_{N^c}(t_{N,2j}-k)+1}\right)
\cdot\frac{\gamma_{N^c}(t_{N,2j}-k)+1}{t_{N,2j+1}-k+1}\\
=&\left(\frac{\zeta_{\gamma_{N^c}(t_{N,2j}-k)+1}(N_2)}{\gamma_{N^c}(t_{N,2j}-k)+1}
-\frac{\zeta_{\gamma_{N^c}(t_{N,2j}-k)+1}(N_7)}{\gamma_{N^c}(t_{N,2j}-k)+1}\right)
\cdot\frac{\zeta_{t_{N,2j}-k+1}(N^c)}{t_{N,2j+1}-k+1}\\
\rightarrow&K^{-k}(1-q)\text{ while }i\rightarrow\infty\,\,(\text{since } \mu(N_2)=K^{-k}, \mu(N_7)=0 \text{ and } \frac{\zeta_{t_{N,2j+1}}(N^c)}{t_{N,2j+1}}\rightarrow 1-q).
\end{aligned}
$$

$$\begin{aligned}
&\frac{\zeta_{t_{N,2j+2}-k+1}(N_4\sqcup(N_1\cap N_6))}{t_{N,2j+2}-k+1}\\
=&\frac{\zeta_{t_{N,2j+2}-k+1}(N_4)}{t_{N,2j+2}-k+1}
+\frac{\zeta_{t_{N,2j+2}-k+1}(N_1\cap N_6)}{t_{N,2j+2}-k+1}\\
=&\frac{\zeta_{t_{N,2j+2}-k+1}(N_4)}{t_{N,2j+2}-k+1}
+\frac{\zeta_{t_{N,2j+2}-k+1}(\gamma_{N^c}^{-1}(N_2\cap N_8))}{t_{N,2j+2}-k+1}\text{ (by (\ref{e:6.10}))},
\end{aligned}$$
where
$$\begin{aligned}
&\frac{\zeta_{t_{N,2j+2}-k+1}(N_4)}{t_{N,2j+2}-k+1}\\
=&\left(\frac{\zeta_{t_{N,2j+2}}(N)}{t_{N,2j+2}}
-\frac{\zeta_{t_{N,2j+2}}(N_3)}{t_{N,2j+2}}\right)
\cdot\frac{t_{N,2j+2}}{t_{N,2j+2}-k+1}
\text{ (since $[t_{N,2j+2}-k+1,t_{N,2j+2})\cap\mathbb{N}\subset N^c$)}\\
\rightarrow&p\text{ while }i\rightarrow\infty\text{ (since $\frac{\zeta_{t_{N,2j+2}}(N)}{t_{N,2j+2}}\rightarrow p$ and $\mu(N_3)=0$)}
\end{aligned}$$
and
$$
\begin{aligned}
&\frac{\zeta_{t_{N,2j+2}-k+1}(\gamma_{N^c}^{-1}(N_2\cap N_8))}{t_{N,2j+2}-k+1}\\
=&\frac{\zeta_{\gamma_{N^c}(t_{N,2j+2}-k)+1}(N_2\cap N_8)}{t_{N,2j+2}-k+1}\,\,(\text{since } \gamma_{N^c}(t_{N,2j+2}-k)\in N_8)\\
=&\frac{\zeta_{\gamma_{N^c}(t_{N,2j+2}-k)+1}(N_2)}{t_{N,2j+2}-k+1}
-\frac{\zeta_{\gamma_{N^c}(t_{N,2j+2}-k)+1}(N_7)}{t_{N,2j+2}-k+1}\\
=&\left(\frac{\zeta_{\gamma_{N^c}(t_{N,2j+2}-k)+1}(N_2)}{\gamma_{N^c}(t_{N,2j+2}-k)+1}
-\frac{\zeta_{\gamma_{N^c}(t_{N,2j+2}-k)+1}(N_7)}{\gamma_{N^c}(t_{N,2j+2}-k)+1}\right)
\cdot\frac{\gamma_{N^c}(t_{N,2j+2}-k)+1}{t_{N,2j+2}-k+1}\\
=&\left(\frac{\zeta_{\gamma_{N^c}(t_{N,2j+2}-k)+1}(N_2)}{\gamma_{N^c}(t_{N,2j+2}-k)+1}
-\frac{\zeta_{\gamma_{N^c}(t_{N,2j+2}-k)+1}(N_7)}{\gamma_{N^c}(t_{N,2j+2}-k)+1}\right)
\cdot\frac{\zeta_{t_{N,2j+2}-k+1}(N^c)}{t_{N,2j+2}-k+1}\\
\rightarrow&K^{-k}(1-p)\text{ while }i\rightarrow\infty\,\,(\text{since } \mu(N_2)=K^{-k}, \mu(N_7)=0 \text{ and } \frac{\zeta_{t_{N,2j+2}}(N^c)}{t_{N,2j+2}}\rightarrow 1-p).
\end{aligned}
$$
Then
\begin{equation}\label{e:6.13}
  \lim_{j\rightarrow\infty}\frac{\zeta_{t_{N,2j}-k+1}(N_4\sqcup(N_1\cap N_6))}{t_{N,2j}-k+1}=p+(1-p)K^{-k}
\end{equation}
and
\begin{equation}\label{e:6.14}
  \lim_{j\rightarrow\infty}\frac{\zeta_{t_{N,2j+1}-k+1}(N_4\sqcup(N_1\cap N_6))}{t_{N,2j+1}-k+1}=q+(1-q)K^{-k}
\end{equation}
So
\begin{equation}\label{e:6.15}
  \inf J\leq p+(1-p)K^{-k},\,\,\sup J\ge q+(1-q)K^{-k}.
\end{equation}

Let
$$J_0=\omega\left(\frac{\zeta_n(N_1\cap(N_4\sqcup N_6))}{n}:n\in N_4+1\right)$$
and
$$J_1=\omega\left(\frac{\zeta_n(N_1\cap(N_4\sqcup N_6))}{n}:n\in N_6+1\right)$$
Then
\begin{equation}\label{e:6.16}
  J=J_0\cup J_1.
\end{equation}

Suppose $t_{N,2j,r}\in N_4$ with $j\ge j^*+1$. Then $t_{N,2j}\leq t_{N,2j,r}<t_{N,2j+1}-k+1$.
We have
\begin{equation*}
  \begin{aligned}
    &\frac{\zeta_{t_{N,2j+1}-k+1}(N_4\sqcup(N_1\cap N_6))}{t_{N,2j+1}-k+1}\\
    \geq &\frac{\zeta_{t_{N,2j+1}-k+1}(N_4\sqcup(N_1\cap N_6))-((t_{N,2j+1}-k+1)-(t_{N,2j}+r+1))}{t_{N,2j+1}-k+1-((t_{N,2j+1}-k+1)-(t_{N,2j}+r+1))}\\
    =&\frac{\zeta_{t_{N,2j,r}+1}(N_4\sqcup(N_1\cap N_6))}{t_{N,2j,r}+1}\,\,(\text{since }[t_{N,2j}+r+1,t_{N,2j}-k+1)\cap\mathbb{N}\subset N_4)\\
    =&\frac{\zeta_{t_{N,2j}}(N_4\sqcup(N_1\cap N_6))+r+1}{t_{N,2j}+r+1}\,\,(\text{since }[t_{N,2j},t_{N,2j}+r+1)\cap\mathbb{N}\subset N_4)\\
    \geq&\frac{\zeta_{t_{N,2j}}(N_4\sqcup(N_1\cap N_6))}{t_{N,2j}}\\
    =&\frac{\zeta_{t_{N,2j}-k+1}(N_4\sqcup(N_1\cap N_6))}{t_{N,2j}}\,\,(\text{since }[t_{N,2j}-k+1,t_{N,2j})\cap (N_4\sqcup N_6)=\emptyset)\\
    =&\frac{\zeta_{t_{N,2j}-k+1}(N_4\sqcup(N_1\cap N_6))}{t_{N,2j}-k+1}
    \cdot\frac{t_{N,2j}-k+1}{t_{N,2j}}.
  \end{aligned}
\end{equation*}
That is,
\begin{equation}\label{e:6.17}
  \begin{aligned}
    &\frac{\zeta_{t_{N,2j}-k+1}(N_4\sqcup(N_1\cap N_6))}{t_{N,2j}-k+1}
    \cdot\frac{t_{N,2j}-k+1}{t_{N,2j}}\\
    \leq&\frac{\zeta_{t_{N,2j,r}+1}(N_4\sqcup(N_1\cap N_6))}{t_{N,2j,r}+1}\\
    \leq&\frac{\zeta_{t_{N,2j+1}-k+1}(N_4\sqcup(N_1\cap N_6))}{t_{N,2j+1}-k+1}.
  \end{aligned}
\end{equation}
It follows from (\ref{e:6.17}), (\ref{e:6.13}) and (\ref{e:6.14}) that
\begin{equation}\label{e:6.18}
  J_0\subset [p+(1-p)K^{-k},q+(1-q)K^{-k}].
\end{equation}

Suppose $t_{N,2j+1,r}\in N_6$ with $j\ge j^*$. Then $t_{N,2j+1}\leq t_{N,2j+1,r}<t_{N,2j+2}-k+1$.
We have
\begin{equation*}
  \begin{aligned}
  &\frac{\zeta_{t_{N,2j+1,r}+1}(N_4\sqcup(N_1\cap N_6))}{t_{N,2j+1,r}+1}\\
  =&\frac{\zeta_{t_{N,2j+1}-k+1}(N_4)
  +\zeta_{t_{N,2j+1,r}+1}(N_1\cap N_6)}{t_{N,2j+1}+r+1}\\
  =&\frac{\zeta_{t_{N,2j+1}-k+1}(N_4)}{t_{N,2j+1}+r+1}
  +\frac{\zeta_{\gamma_{N^c}(t_{N,2j+1,r})+1}(N_2\cap N_8)}{\gamma_{N^c}(t_{N,2j+1,r})+1}
  \cdot\frac{\gamma_{N^c}(t_{N,2j+1,r})+1}{t_{N,2j+1}+r+1}\\
  =&\frac{\zeta_{t_{N,2j+1}-k+1}(N_4)}{t_{N,2j+1}+r+1}
  +\frac{\zeta_{\gamma_{N^c}(t_{N,2j+1,r})+1}(N_2\cap N_8)}{\gamma_{N^c}(t_{N,2j+1,r})+1}
  \cdot\frac{\zeta_{t_{N,2j+1,r}}(N^c)+1}{t_{N,2j+1}+r+1}\\
  =&\frac{\zeta_{t_{N,2j+1}-k+1}(N_4)}{t_{N,2j+1}+r+1}
  +\frac{\zeta_{\gamma_{N^c}(t_{N,2j+1,r})+1}(N_2\cap N_8)}{\gamma_{N^c}(t_{N,2j+1,r})+1}
  \cdot\frac{\zeta_{t_{N,2j+1}}(N^c)+r+1}{t_{N,2j+1}+r+1}\\
  &(\text{since }[t_{N,2j+1},t_{N,2j+1,r})\cap\mathbb{N}\subset N^c)
  \end{aligned}
\end{equation*}
That is,
\begin{equation}\label{e:6.19}
  \begin{aligned}
  &\frac{\zeta_{t_{N,2j+1,r}+1}(N_4\sqcup(N_1\cap N_6))}{t_{N,2j+1,r}+1}\\
  =&\frac{\zeta_{t_{N,2j+1}-k+1}(N_4)}{t_{N,2j+1}+r+1}
  +\frac{\zeta_{\gamma_{N^c}(t_{N,2j+1,r})+1}(N_2\cap N_8)}{\gamma_{N^c}(t_{N,2j+1,r})+1}
  \cdot\frac{\zeta_{t_{N,2j+1}}(N^c)+r+1}{t_{N,2j+1}+r+1}
  \end{aligned}
\end{equation}

Let $\epsilon>0$. Define the function
$$g:\left[0,\frac{q}{p}-1+\epsilon\right]\rightarrow\mathbb{R},\,\,g(s)=\frac{q+K^{-k}(1-q+s)}{1+s}.$$
Since $q>0$, $g$ is well defined (and, in the case $p=0$, $g(+\infty)=K^{-k}$).
Note that $g(s)$ is non-increasing and hence
\begin{equation}\label{e:6.20}
  \begin{aligned}
    g\left(\left[0,\frac{q}{p}-1+\epsilon\right]\right)
    &=\left[g\left(\frac{q}{p}-1+\epsilon\right),g(0)\right]\\
    &=\left[\frac{pq+K^{-k}(1-p)q+K^{-k}p\epsilon}{q+p\epsilon},q+K^{-k}(1-q)\right].
  \end{aligned}
\end{equation}
Since $q>0$, the limit
$$\frac{t_{N,2j+2}}{t_{N,2j+1}}
=\frac{\frac{e_{N,2j}}{t_{N,2j+1}}}{\frac{e_{N,2j}}{t_{N,2j+2}}}
=\frac{\frac{e_{N,2j}}{t_{N,2j+1}}}{\frac{e_{N,2j+1}}{t_{N,2j+2}}}
\rightarrow \frac{q}{p},\,\,\rightarrow\infty$$
holds (and equals $+\infty$ when $p=0$). Then
$$\frac{d_{N,2j+1}}{t_{N,2j+1}}\rightarrow \frac{q}{p}-1,\,\,j\rightarrow\infty.$$
So we can take $j^{**}\geq j^*$ such that if $j\geq j^{**}+1$, then
$$\frac{d_{N,2j+1}}{t_{N,2j+1}}\rightarrow \frac{q}{p}-1+\epsilon.$$
Let $j\ge j^{**}+1$ and $t_{N,2j+1,r}\in N_6$.
Then $t_{N,2j+1,r}\in[t_{N,2j+1},t_{N,2j+2}-k+1)$.
Write
$$s=\frac{r+1}{t_{N,2j+1}}.$$
Then $0<s<\frac{q}{p}-1+\epsilon$. Using (\ref{e:6.19}) we get
\begin{equation*}
  \begin{aligned}
    &\frac{\zeta_{t_{N,2j+1,r}+1}(N_4\sqcup(N_1\cap N_6))}{t_{N,2j+1,r}+1}-g(s)\\
  =&\left(
  \frac{
    \frac{
        \zeta_{t_{N,2j+1}-k+1}(N_4)
    }{t_{N,2j+1}}
  }{1+s}
  -\frac{q}{1+s}\right)\\
  &+\left(\frac{\zeta_{\gamma_{N^c}(t_{N,2j+1,r})+1}(N_2\cap N_8)}{\gamma_{N^c}(t_{N,2j+1,r})+1}
  \cdot
  \frac{
    \frac{
        \zeta_{t_{N,2j+1}}(N^c)+r+1
    }{t_{N,2j+1}}+s
  }{1+s}
  -K^{-k}\cdot\frac{1-q+s}{1+s}\right)\\
  \rightarrow&0\text{ independent from }r,\,\,\text{ while }j\rightarrow\infty.
  \end{aligned}
\end{equation*}
This limit together with (\ref{e:6.20}) lead to
\begin{equation*}
  J_1\subset\left[\frac{pq+K^{-k}(1-p)q+K^{-k}p\epsilon}{q+p\epsilon},q+K^{-k}(1-q)\right].
\end{equation*}
Letting $\epsilon\rightarrow0$ we get
\begin{equation}\label{e:6.21}
  J_1\subset[p+K^{-k}(1-p),q+K^{-k}(1-q)].
\end{equation}
Now (\ref{e:6.6}) follows from (\ref{e:6.18}), (\ref{e:6.21}), (\ref{e:6.16}) and (\ref{e:6.15}).
\end{proof}

For convenience of our later use, we construct $N\in\mathcal{M}([p,q])$ in Example \ref{l:2.24} below.
Our constructions are similar to those in \cite{WS}.
\begin{example}\label{l:2.24}
  Let $[p,q]\in\mathcal{C}([0,1])$. We choose a $\delta\in(0,1)$, set $c_0=0$ and $c_1=1$, and choose a $c_i\in(0,1)$, $i\geq 2$, such that
  \begin{equation}\label{e:2.22}
    \text{for }i\ge 1,\,\,c_{2i+1}=\left\{
    \begin{aligned}
    &\frac{1}{\sqrt{2i+1}},&\text{if }&q=0,\\
    &q-\frac{q}{\sqrt{2i+1}},&\text{if }&0<q\leq 1,
    \end{aligned}
    \right.
  \end{equation}
  and
  \begin{equation}\label{e:2.23}
    c_{2i+2}<c_{2i+1}-\frac{\delta}{\sqrt{2i+2}},\,\,i\ge 0,\text{ and }c_{2i}\rightarrow p,\,\,i\rightarrow\infty.
  \end{equation}
  Set $t_0=0$ and $t_1=1$. Define iteratively $t_{2i+2}$ to be the least integer satisfying $t_{2i+2}\geq t_{2i+1}+1$ and
  \begin{equation}\label{e:2.24}
    \frac{\sum_{0\leq 2j<2i+2}(t_{2j+1}-t_{2j})}{t_{2i+2}}<c_{2i+2}
  \end{equation}
  and $t_{2i+3}$ the least integer satisfying $t_{2i+3}\ge t_{2i+2}+1$ and
  \begin{equation}\label{e:2.25}
    \frac{\sum_{0\leq 2j<2i+3}(t_{2j+1}-t_{2j})}{t_{2i+3}}\ge c_{2i+3}.
  \end{equation}
  Let
  $$N=\left(\bigcup_{j\ge 0}[t_{2j},t_{2j+1})\right)\cap\mathbb{N}.$$
Then for $j\ge 0$, $t_{N,j}=t_j$. By $(\ref{e:2.25})$, $(\ref{e:2.22})$, $(\ref{e:2.24})$ and $(\ref{e:2.23})$,
  \begin{equation}\label{e:2.26}
    \frac{e_{N,2j}}{t_{N,2j+1}}\rightarrow q \text{ and }\frac{e_{N,2j+1}}{t_{N,2j+2}}\rightarrow p\text{ while }j\rightarrow\infty.
  \end{equation}
and
\begin{equation}\label{e:2.27}
\left\{
  \begin{aligned}
&\frac{e_{N,2j+2}}{t_{N,2j+3}}
=\frac{e_{N,2j+1}++d_{N,2j+2}}{t_{N,2j+2}+d_{N,2j+2}}
\geq c_{2j+3},\\
&\frac{e_{N,2j+1}}{t_{N,2j+2}}
<c_{2j+1}-\frac{\delta}{\sqrt{2j+2}},
  \end{aligned}
\right.
\end{equation}
\begin{equation}\label{e:2.28}
\left\{
  \begin{aligned}
&\frac{f_{N,2j+1}}{t_{N,2j+2}}
=\frac{f_{N,2j+1}+d_{N,2j+1}}{t_{N,2j+1}+d_{N,2j+1}}
>1-c_{2j+1}+\frac{\delta}{\sqrt{2j+2}},\\
&\frac{f_{N,2j}}{t_{N,2j+1}}
\leq 1-c_{2j+1}.
  \end{aligned}
\right.
\end{equation}
From $(\ref{e:2.27})$ and $(\ref{e:2.22})$ we get
\begin{equation}\label{e:2.29}
  \begin{aligned}
d_{N,2j+2}&>\frac{c_{2j+3}-c_{2j+1}+\frac{\delta}{\sqrt{2j+2}}}{1-c_{2j+3}}\cdot t_{N,2j+2}\\
&\geq \frac{\frac{\delta}{\sqrt{2j+2}}
-\left|\frac{1}{\sqrt{2j+3}}
-\frac{1}{\sqrt{2j+1}}\right|}{1}\cdot(2j+2)\\
&\rightarrow \infty,\,\,j\rightarrow\infty.
  \end{aligned}
\end{equation}
From $(\ref{e:2.28})$ we get
\begin{equation}\label{e:2.30}
  \begin{aligned}
d_{N,2j+1}&>\frac{\frac{\delta}{\sqrt{2j+2}}}{1-c_{2j+1}
+\frac{\delta}{\sqrt{2j+2}}}\cdot t_{N,2j+1}\\
&\geq\frac{\frac{\delta}{\sqrt{2j+2}}}{1+\delta}\cdot (2j+1)\\
&\rightarrow\infty,\,\,j\rightarrow\infty.
  \end{aligned}
\end{equation}
By $(\ref{e:2.26})$, $(\ref{e:2.29})$ and $(\ref{e:2.30})$, $N\in\mathcal{M}([p,q])$.
Suppose $q>0$. Since
$$e_{N,2i}-1<c_{2i+1}(t_{N,2i+1}-1)
=\left(q-\frac{q}{\sqrt{2i+1}}\right)(t_{N,2i+1}-1)$$
for large $i$, then, for large $i$,
$$\begin{aligned}
&f_{N,2i}-(1-q)t_{N,2i+1}\\
=&f_{N,2i}-t_{N,2i+1}+qt_{N,2i+1}\\
=&qt_{N,2i+1}-e_{N,2i}\\
>&qt_{N,2i+1}-\left(q-\frac{q}{\sqrt{2i+1}}\right)(t_{N,2i+1}-1)-1\\
=&\frac{q}{\sqrt{2i+1}}\cdot t_{N,2i+1}+\left(q-\frac{q}{\sqrt{2i+1}}\right)-1\\
\geq&\frac{q}{\sqrt{2i+1}}(2i+1)-1\\
\rightarrow&+\infty,\,\,i\rightarrow\infty.
\end{aligned}$$
That is,
\begin{equation}\label{e:2.31}
  \lim_{i\rightarrow\infty}(f_{N,2i}-(1-q)t_{N,2i+1})=+\infty\text{ while }q>0.
\end{equation}
\end{example}

\begin{lemma}\label{l:6.4}
Let $0<q\leq 1$ and $0\leq p\leq q$. Then $\mathscr{H}^{2-q}(E_{\sigma_K}([p,q]))=+\infty$.
\end{lemma}
\begin{proof}
Let $G=G_{\mathscr{H}^1,\sigma_{K^2}}$. Let $N\in\mathcal{M}([0,1])$ be as in Example \ref{l:2.24}.
Define
$$X=\{\Phi_{N,x}(y):x\in E_K^\mathbb{N}, y\in G\}.$$
Lemma \ref{l:6.3} implies that  $\pi_K(E_{\sigma_K}([p,q]))\supset X$. Then, by Lemma \ref{detosw}, to prove $\mathscr{H}^{2-q}(E_{\sigma_K}([p,q]))=+\infty$ it is enough to prove $\mathscr{H}^{1-\frac{q}{2}}(X)=+\infty$.

For $i\geq0$ let
$$V_i=\{v\in \Sigma_{K,i}:[v]\cap X\neq\emptyset\}.$$
Since $G$ is dense in $\Sigma_{K^2}$, by the definition of $X$, we can check
\begin{equation}\label{e:6.22}
  V_i=\{v\in \Sigma_{K,i}:v_j\in E_K\text{ for }j\in N\cap \{0,1,\cdots,i-1\}\}.
\end{equation}

Suppose $\omega\in V_{t_{N,2j+1,r}+1}$. Define the set
$$G_\omega=\{y\in G:y_i=\omega_{\gamma_{N^c}^{-1}(i)}\text{ for }i\in\{0,1,\cdots,\gamma_{N^c}(t_{N,2j+1,r})\}\}.$$
Let $v\in W_k$ be the longest word with $G_\omega\subset[v]$. Then
$$|v|=\gamma_{N^c}(t_{N,2j+1,r})+1=\zeta_{t_{N,2j+1,r}+1}(N^c)=f_{N,2j}+r+1.$$
So
\begin{equation}\label{e:6.23}
\mathscr{H}^1(G_\omega)=\mathscr{H}^1(G\cap[v])=\mathscr{H}^1([v])=(K^2)^{-(f_{N,2j}+r+1)}.
\end{equation}
Define the function $h_\omega:G\rightarrow\mathbb{R}$ by
$$h_\omega(y)=\left\{
\begin{aligned}
  &K^{-e_{N,2j}},&&y\in G_\omega,\\
  &0,&&\text{otherwise}.
\end{aligned}
\right.
$$
Clearly $h_\omega$ is a continuous function on $G$ and, by (\ref{e:6.23}),
\begin{equation}\label{e:6.24}
  \int_G h_\omega \mathrm{d} \mathscr{H}^1=K^{-e_{N,2j}}(K^2)^{-(f_{N,2j}+r+1)}.
\end{equation}
Moreover,
\begin{equation}\label{e:6.25}
  \begin{aligned}
    |[\omega]|^{1-\frac{q}{2}}&=(K^2)^{-(t_{N,2j+1,r}+1)(1-\frac{q}{2})}\\
    &=(K^2)^{-(e_{N,2j}+f_{N,2j}+r+1)(1-\frac{q}{2})}\\
    &=K^{-e_{N,2j}}(K^2)^{-(f_{N,2j}+r+1)}K^{f_{N,2j}q-e_{N,2j}(1-q)+(r+1)q}\\
    &\geq K^{-e_{N,2j}}(K^2)^{-(f_{N,2j}+r+1)}K^{f_{N,2j}q-e_{N,2j}(1-q)}\\
    &=K^{f_{N,2j}q-e_{N,2j}(1-q)}\int_G h_\omega \mathrm{d} \mathscr{H}^1\\
    &=K^{t_{N,2j+1}q-e_{N,2j}}\int_G h_\omega \mathrm{d} \mathscr{H}^1\\
    &=K^{f_{N,2j}-t_{N,2j+1}(1-q)}\int_G h_\omega \mathrm{d} \mathscr{H}^1.
  \end{aligned}
\end{equation}

Let $M>0$. Because of (\ref{e:2.31}), we may choose $j^*$ such that
\begin{equation}\label{e:6.26}
  K^{f_{N,2j}-t_{N,2j+1}(1-q)}\geq M,\,\,j\geq j^*.
\end{equation}
Then, for $\omega\in V_{t_{N,2j+1,r}+1}$, it follows from (\ref{e:6.25}) and (\ref{e:6.26}) that
\begin{equation}\label{e:6.27}
  |[\omega]|^{1-\frac{q}{2}}\geq M\int_G h_\omega \mathrm{d} \mathscr{H}^1,\,\,j\geq j^*.
\end{equation}

Let $k\geq t_{N,2j^*+1}$, $\mathcal{B}\in \mathcal{C}_{X,(K^2)^{-k}}$ and $\epsilon>0$.
Since $X$ is perfect, then we may choose $\mathcal{B}_0\in\mathscr{C}_{X,K^{-k}}$ such that for each $B_0\in\mathcal{B}_0$ we have $\#(B_0)\geq 2$, for each $B\in\mathcal{B}_0$ with $B\subset B_0$ and
\begin{equation}\label{e:6.28}
  \sum_{B\in\mathcal{B}_0}|B|^{1-\frac{q}{2}}<\sum_{B\in \mathcal{B}}|B|^{1-\frac{q}{2}}+\epsilon.
\end{equation}

Let $\mathcal{B}_1=\{[\omega_B]\cap X:B\in\mathcal{B}_0\}$, where $[\omega_B]$ is the longest column in $\Sigma_K$ that contains $B$.
Since each $|[\omega_B]|=|B|$,
\begin{equation}\label{e:6.29}
  \sum_{B\in \mathcal{B}_1}|B|^{1-\frac{q}{2}}\leq\sum_{B\in \mathcal{B}_0}|B|^{1-\frac{q}{2}}.
\end{equation}
Note that for any $v,\omega\in\{\omega_B:B\in \mathcal{B}_0\}$, one of the three statements
$$[v]\subset[\omega],\,\,[\omega]\subset[v],\,\,[v]=[\omega]$$
is true.
Then we may choose a subcover $\mathcal{B}_2\subset\mathcal{B}_1$ with $[\omega_B]$, $B\in \mathcal{B}_2$, pairwise disjoint.
Then
\begin{equation}\label{e:6.30}
  \sum_{B\in\mathcal{B}_2}|B|^{1-\frac{q}{2}}\leq\sum_{B\in\mathcal{B}_1}|B|^{1-\frac{q}{2}}
\end{equation}

Suppose $B\in\mathcal{B}_2$ with $|\omega_B|=t_{N,2j,r}+1\in N+1$. Put
$$\mathcal{C}_B=\{[\omega]\cap X:\omega\in V_{t_{N,2j,r}+1},\omega|_{[0,t_{N,2j,r}+1)}=\omega_B\}.$$
Then $\bigcup\mathcal{C}_B=B$. By (\ref{e:6.22}), it follows that $\mathcal{C}_B$ contains
$K^{d_{N,2j}-r}$ pairwise
disjoint members of diameter $(K^2)^{-t_{N,2j+1}}$.
Then
\begin{equation}\label{e:6.31}
\begin{aligned}
  \sum_{C\in\mathcal{C}_B}|C|^{1-\frac{q}{2}}&=K^{d_{N,2j}-r}(K^2)^{-t_{N,2j+1}(1-\frac{q}{2})}\\
  &=K^{d_{N,2j}-r}(K^2)^{-t_{N,2j+1,r}(1-\frac{q}{2})}(K^2)^{-(d_{N,2j}-r)(1-\frac{q}{2})}\\
  &=(K^2)^{-t_{N,2j+1,r}(1-\frac{q}{2})}(K^2)^{-(d_{N,2j}-r)(\frac{1}{2}-\frac{q}{2})}\\
  &\leq(K^2)^{-t_{N,2j+1,r}(1-\frac{q}{2})}\\
  &=|B|^{1-\frac{q}{2}}.
\end{aligned}
\end{equation}

Let
$$\mathcal{B}_3=\{C:C\in\mathcal{C}_B,\,\,B\in\mathcal{B}_2,\,\,|\omega_B|\in N+1\}\cup\{B:B\in \mathcal{B}_2,\,\,|\omega_B|\in N^c+1\}.$$
Then $\mathcal{B}_3$ is a cover of $X$ and a set of pairwise disjoint subsets of $X$, $|\omega_B|\in N^c+1$ for $B\in \mathcal{B}_3$ and, by (\ref{e:6.31}),
\begin{equation}\label{e:6.32}
  \sum_{B\in \mathcal{B}_3}|B|^{1-\frac{q}{2}}\leq \sum_{B\in \mathcal{B}_2}|B|^{1-\frac{q}{2}}
\end{equation}

Let $j_0\leq j_1$. Suppose $y\in G$ and $\omega\in V_{t_{N,2j_0+1,r+1}}$ such that
$$y_i=\omega_{\gamma^{-1}_{N^c}(i)}\text{ for }i\in\{0,1,\cdots,\gamma_{N^c}(t_{N,2j_0+1,r})\}$$
and $V_\omega$ is a maximal subset of $\bigcup_{t_{N,2j_1+1}+1\leq t \leq t_{N,2j_1+2}+1}V_t$ such that
\begin{description}
  \item{(i)} for $v\in V_\omega$, $|v|\geq |\omega|$ and $v|_{[0,|\omega|)}=\omega$,
  \item{(ii)} for $v\in V_\omega$,
  $$y_i=\omega_{\gamma^{-1}_{N^c}(i)}\text{ for }i\in\{0,1,\cdots,\gamma_{N^c}(|v|)\},$$
  \item{(iii)} for $v_0\neq v_1\in V_\omega$, $[v_0]\cap[v_1]=\emptyset$.
\end{description}
Then
$$\#(V_\omega)=K^{\#(N\cap[t_{N,2j_0+2},t_{N,2j_1+1}))}=K^{e_{N,2j_1}-e_{N,2j_0}}.$$
Thus
\begin{equation}\label{e:6.33}
  \sum_{v\in V_\omega}h_v(y)
  =K^{e_{N,2j_1}-e_{N,2j_0}}K^{-e_{N,2j_1}}
  =K^{-e_{N,2j_0}}=h_\omega(y).
\end{equation}

Let $y\in G$. Put $X_y=\{\Phi_{N,x}(y):x\in E_K^\mathbb{N}\}$ and $\mathcal{C}_y=\{B\in \mathcal{B}_3:B\cap X_y\neq\emptyset\}$.
Since $X_y$ is compact, $\mathcal{C}_y$ covers $X_y$ and the members of $\mathcal{C}_y$ are open in $X$ and are pairwise disjoint. Hence $\mathcal{C}_y$ is finite.
Suppose $\sup\{|\omega_B|:B\in\mathcal{C}_x\}\in [t_{N,2j_1+1}+1,t_{N,2j_1+2}+1)$.
For each $\omega\in \{\omega_B:B\in \mathcal{C}_x\}$ choose $V_\omega$ to be a maximal subset of $\bigcup_{t_{N,2j_1+1}+1\leq t \leq t_{N,2j_1+2}+1}V_t$ satisfying (i), (ii) and (iii).
Using (\ref{e:6.33}) we have
\begin{equation}\label{e:6.34}
  \sum_{B\in \mathcal{C}_x}h_{\omega_B}(y)=\sum_{B\in \mathcal{C}_x}\sum_{v\in V_{\omega_B}}h_v(y)=K^{e_{N,2j_1}}K^{-e_{N,2j_1}}=1.
\end{equation}
Eq. (\ref{e:6.34}) leads to
\begin{equation}\label{e:6.35}
  \sum_{B\in\mathcal{B}_3}h_{\omega_B}(y)\equiv 1\text{ for }y\in G.
\end{equation}
Note that, for $B\in\mathcal{B}_3$, $|\omega_B|\geq t_{N,2j^*+1}+1$.
Then it follows from (\ref{e:6.27}) and (\ref{e:6.35}) that
\begin{equation}\label{e:6.36}
\begin{aligned}
  \sum_{B\in\mathcal{B}_3}|B|^{1-\frac{q}{2}}&=\sum_{B\in\mathcal{B}_3}|[\omega_B]|^{1-\frac{q}{2}}\geq M\int_G h_\omega \mathrm{d} \mathscr{H}^1\\
  &=M\int_G \left(\sum_{B\in\mathcal{B}_3}h_{\omega_B}\right) \mathrm{d} \mathscr{H}^1=M\mathscr{H}^1(G)=M.
  \end{aligned}
\end{equation}

By (\ref{e:6.28}), (\ref{e:6.29}), (\ref{e:6.30}), (\ref{e:6.32}) and (\ref{e:6.36}),
\begin{equation}\label{e:6.37}
  \sum_{B\in\mathcal{B}}|B|^{1-\frac{q}{2}}\geq M-\epsilon.
\end{equation}
Since (\ref{e:6.37}) holds for each $B\in \mathscr{C}_{X,(K^2)^{-k}}$,
\begin{equation}\label{e:6.38}
  \mathscr{H}_{(K^2)^{-k}}^{1-\frac{q}{2}}(X)\geq M-\epsilon.
\end{equation}
Letting $k\rightarrow\infty$ in (\ref{e:6.38}) and then letting $M\rightarrow\infty$, we get $\mathscr{H}^{1-\frac{q}{2}}(X)=+\infty$.
\end{proof}

We sum up Lemma \ref{l:6.2} and Lemma \ref{l:6.4} into the following theorem.
\begin{theorem}\label{l:6.5}
$\mathscr{H}^2(E_{\sigma_K}([0,0]))=1$ and, for $0<q\leq 1$ and $0\leq p\leq q$, $\mathscr{H}^{2-q}(E_{\sigma_K}([p,q]))=+\infty$.
\end{theorem}

To calculate the Hausdorff measure of $D_{\sigma_K}([p,q])$, we need some lemmas.
\begin{lemma}\label{l:2.20}
Suppose $0\leq l<r$, $n\geq 1$ and $0\leq j_i<n$, $i=0,1$. Then
$$|\#((n\mathbb{N}+j_0)\cap[l,r))-\#((n\mathbb{N}+j_1)\cap[l,r))|\leq 1.$$
\end{lemma}
\begin{proof}
Let $k$ be the maximal integer satisfying $l+kn\leq r$. Then
$$\#((n\mathbb{N}+j_0)\cap[l,l+kn))=\#((n\mathbb{N}+j_1)\cap[l,l+kn))$$
and
$$\#((n\mathbb{N}+j_0)\cap[l+kn,r)),\#((n\mathbb{N}+j_1)\cap[l+kn,r))\in\{0,1\}.$$
\end{proof}

\begin{lemma}\label{l:2.21}
  Let $N\subset\mathbb{N}$ with both $N$ and $N^c$ infinite. Suppose $t_k\rightarrow\infty$ satisfy
  \begin{equation}\label{e:2.16}
  \lim_{k\rightarrow\infty}\frac{\zeta_{t_k}(N)}{t_k}=p
  \text{ and }\lim_{k\rightarrow\infty}\frac{\zeta_{t_k}(L_N)}{t_k}=0.
  \end{equation}
  Then, for $n\geq 1$ and $0\leq j<n$,
  \begin{equation}\label{e:2.17}
  \lim_{k\rightarrow\infty}\frac{\zeta_{t_k}(N\cap(n\mathbb{N}+j))}{t_k}=\frac{p}{n}.
  \end{equation}
\end{lemma}
\begin{proof}
Let $n\geq 1$ and $0\leq j<n$. Note that $N\cap\{0,1,\cdots,t_k-1\}$ is the union of $\zeta_{t_k}(L_N)$ integer intervals. Then, for $0\leq j'<n$, by Lemma \ref{l:2.20},
\begin{equation}\label{e:2.18}
|\zeta_{t_k}(N\cap(n\mathbb{N}+j))-\zeta_{t_k}(N\cap(n\mathbb{N}+j'))|\leq\zeta_{t_k}(L_N).
\end{equation}
Sum (\ref{e:2.18}) over $0\leq j'<n$ and divide the resulting formula by $nt_k$. We get
\begin{equation}\label{e:2.19}
\left|\frac{\zeta_{t_k}(N\cap(n\mathbb{N}+j))}{t_k}-\frac{\zeta_{t_k}(N)}{nt_k}\right|\leq\frac{\zeta_{t_k}(L_N)}{nt_k}.
\end{equation}
Now (\ref{e:2.17}) follows from (\ref{e:2.19}) and (\ref{e:2.16}).
\end{proof}

\begin{lemma}\label{l:6.6}
For $0\leq p\leq q<1$, $\mathscr{H}^{2-q}(D_{\sigma_K}([p,q]))=0$.
\end{lemma}
\begin{proof}
Suppose $0\leq p\leq q<1$. Let $X=\pi_K(D_{\sigma_K}([p,q]))\subset \Sigma_{K^2}$. Then, by Lemma \ref{detosx}, to prove $\mathscr{H}^{2-q}(D_{\sigma_K}([p,q]))=0$ it is enough to prove $\mathscr{H}^{1-\frac{q}{2}}(X)=0$.

For $n\geq 1 $ let
\begin{equation}\label{e:6.39}
  X_n=\{x\in X:\mu^*(N_{\sigma_{K^2}}(x,[E_K^k]))\equiv q\text{ for }k\geq n\}.
\end{equation}
Then
\begin{equation}\label{e:6.40}
  X_n\subset X_{n+1}\text{ for } n\geq 1 \text{ and }X=\bigcup_{n\geq 1}X_n.
\end{equation}

Fix $n\geq 1$ and $x\in X_n$. For $k\geq 1$ write
$$N_k=N_{\sigma_{K^2}}(x,[E_K^k])=\{i\geq 0:x|_{[i,i+k)}\in E_K^k\}.$$
For $m\geq k\geq 1$,
$$\begin{aligned}
i\in N_m&\Leftrightarrow x|_{[i,i+m)}\in E_K^m\Leftrightarrow x|_{[i+j,i+j+k)}\in E_K^k\text{ for }j\in\{0,1,\cdots,m-k\}\\
&\Leftrightarrow i+\{0,1,\cdots,m-k\}\subset N_k.\\
\end{aligned}$$
Then it follows from Lemma \ref{detvoe} that
\begin{equation}\label{e:6.41}
  N_m=N_k\setminus (R_{N_k}-\{0,1,\cdots,m-k\})\text{ for }m\geq k\geq1.
\end{equation}

Since $x\in X_n$, by (\ref{e:6.40}) and (\ref{e:6.39}), we may choose $t_k\rightarrow \infty$ such that
\begin{equation}\label{e:6.42}
  \lim_{k\rightarrow\infty}\frac{\zeta_{t_k}(N_k)}{t_k}=q.
\end{equation}
For $m\geq n$, since $N_m\supset N_k$ for $k\geq m$, using (\ref{e:6.42}) we have
$$q\geq \limsup_{k\rightarrow\infty}\frac{\zeta_{t_k}(N_m)}{t_k}
\geq\liminf_{k\rightarrow\infty}\frac{\zeta_{t_k}(N_m)}{t_k}
\geq\lim_{k\rightarrow\infty}\frac{\zeta_{t_k}(N_k)}{t_k}=q.$$
Then
\begin{equation}\label{e:6.43}
  \frac{\zeta_{t_k}(N_m)}{t_k}=q\text{ for }m\geq n.
\end{equation}
By (\ref{e:6.41}), $N_m=N_{m+1}\sqcup(R_{N_m}-1)$. Then
\begin{equation}\label{e:6.44}
  0\leq \frac{\zeta_{t_k}(R_{N_m})}{t_k}
  \leq\frac{\zeta_{t_k}(R_{N_m}-1)}{t_k}
  =\frac{\zeta_{t_k}(N_m)}{t_k}-\frac{\zeta_{t_k}(N_{m+1})}{t_k}
\end{equation}
It follows from (\ref{e:6.44}) and (\ref{e:6.43}) that
\begin{equation}\label{e:6.45}
  \lim_{k\rightarrow\infty}\frac{\zeta_{t_k}(R_{N_m})}{t_k}=0\text{ for }m\geq n.
\end{equation}
Since $\zeta_{t_k}(R_{N_m})\leq\zeta_{t_k}(L_{N_m})\leq\zeta_{t_k}(R_{N_m})+1$, we know from (\ref{e:6.45}) that
\begin{equation}\label{e:6.46}
  \lim_{k\rightarrow\infty}\frac{\zeta_{t_k}(L_{N_m})}{t_k}
  =\lim_{k\rightarrow\infty}\frac{\zeta_{t_k}(R_{N_m})}{t_k}=0\text{ for }m\geq n.
\end{equation}

Let $m\geq n$. It follows from (\ref{e:6.43}), (\ref{e:6.46}) and Lemma \ref{l:2.21} that
\begin{equation}\label{e:6.47}
  \lim_{k\rightarrow\infty}\frac{\zeta_{t_k}(N_m\cap m\mathbb{N})}{t_k}=\frac{q}{m}.
\end{equation}
Note that
$$N_{\sigma_{K^{2m}}}(\tau_{K^2,m}(x),[E_{K,m}])\cap \left[0,\left\lfloor\frac{t_k}{m}\right\rfloor\right)=\frac{N_m\cap m\mathbb{N} \cap[0,t_k)}{m},$$
where $\lfloor a\rfloor$ denotes the maximal integer no larger than $a$ for $a\in\mathbb{R}$. Then
\begin{equation}\label{e:6.48}
  \begin{aligned}
    &\frac
    {
        \zeta_{\left\lfloor\frac{t_k}{m}\right\rfloor}
        (
            N_{\sigma_{K^{2m}}}(\tau_{K^2,m}(x),[E_{K,m}])
        )
    }
    {
        \left\lfloor\frac{t_k}{m}\right\rfloor
    }\\
    =&\frac
    {\#(N_m\cap m\mathbb{N} \cap[0,t_k))}
    {\left\lfloor\frac{t_k}{m}\right\rfloor}\\
    =&\frac
    {\zeta_{t_k}(N_m\cap m\mathbb{N})}
    {\frac{t_k}{m}}\cdot
    \frac
    {\frac{t_k}{m}}
    {\left\lfloor\frac{t_k}{m}\right\rfloor}.
  \end{aligned}
\end{equation}
It follows from (\ref{e:6.48}) and (\ref{e:6.47}) that
\begin{equation}\label{e:6.49}
  \lim_{k\rightarrow\infty}\frac
    {
        \zeta_{\left\lfloor\frac{t_k}{m}\right\rfloor}
        (
            N_{\sigma_{K^{2m}}}(\tau_{K^2,m}(x),[E_{K,m}])
        )
    }
    {
        \left\lfloor\frac{t_k}{m}\right\rfloor
    }
    =q.
\end{equation}
In particular, for $s\geq 1$,
\begin{equation}\label{e:6.50}
   \lim_{k\rightarrow\infty}\frac
    {
        \zeta_{\left\lfloor\frac{t_k}{sn}\right\rfloor}
        (
            N_{\sigma_{K^{2sn}}}(\tau_{K^2,sn}(x),[E_{K,sn}])
        )
    }
    {
        \left\lfloor\frac{t_k}{sn}\right\rfloor
    }
    =q.
\end{equation}

Now put $M_n=N_n^c$. Then, by (\ref{e:6.43}),
\begin{equation}\label{e:6.51}
  \lim_{k\rightarrow\infty}\frac{\zeta_{t_k}(M_n)}{t_k}=1-q.
\end{equation}
Since $R_{M_n}\subset L_{M_n}$ and $L_{M_n}\subset R_{N_n}\cup\{0\}$,
\begin{equation}\label{e:6.52}
  \lim_{k\rightarrow\infty}\frac{\zeta_{t_k}(L_{M_n})}{t_k}
  =\lim_{k\rightarrow\infty}\frac{\zeta_{t_k}(R_{M_n})}{t_k}=0.
\end{equation}
Let $s\geq 1$. Write
$$M_{n,s}=M_n\setminus(R_{M_n}-\{0,1,\cdots,sn-1\})=\{i\geq0:i+\{0,1,\cdots,sn-1\}\subset M_n\}.$$
Since $L_{M_{n,s}}\subset L_{M_n}$ and $R_{M_{n,s}}\subset R_{M_n}-sn+1$, we have by (\ref{e:6.52}) that
\begin{equation}\label{e:6.53}
\lim_{k\rightarrow\infty}\frac{\zeta_{t_k}(L_{M_{n,s}})}{t_k}
=\lim_{k\rightarrow\infty}\frac{\zeta_{t_k}(R_{M_{n,s}})}{t_k}=0.
\end{equation}
Note that $M_{n,s}\subset M_n\subset M_{n,s}\sqcup(R_{M_n}-1-\{0,1,\cdots,sn-1\})$. Then
\begin{equation}\label{e:6.54}
  \begin{aligned}
    \frac{\zeta_{t_k}(M_{n,s})}{t_k}
    &\leq\frac{\zeta_{t_k}(M_n)}{t_k}
    \leq\frac{\zeta_{t_k}(M_{n,s})}{t_k}
    +\frac{\zeta_{t_k}(R_{M_n}-1-\{0,1,\cdots,sn-1\})}{t_k}\\
    &\leq \frac{\zeta_{t_k}(M_{n,s})}{t_k}
    +sn\cdot\frac{\zeta_{t_k}(R_{M_n}-1)}{t_k}.
  \end{aligned}
\end{equation}
Eqs (\ref{e:6.54}), (\ref{e:6.51}) and (\ref{e:6.53}) lead to
\begin{equation}\label{e:6.55}
  \lim_{k\rightarrow\infty}\frac{\zeta_{t_k}(M_{n,s})}{t_k}=1-q.
\end{equation}
It follows from (\ref{e:6.55}), (\ref{e:6.53}) and Lemma \ref{l:2.21} that
\begin{equation}\label{e:6.56}
  \lim_{k\rightarrow\infty}\frac{\zeta_{t_k}(M_{n,s})}{t_k}=\frac{1-q}{sn}.
\end{equation}
Similar to (\ref{e:6.48}), we have
\begin{equation}\label{e:6.57}
  \begin{aligned}
    &\frac
    {
        \zeta_{\left\lfloor\frac{t_k}{sn}\right\rfloor}
        (
            N_{\sigma_{K^{2sn}}}(\tau_{K^2,sn}(x),[\tau_{K^{2n}}(F^s_{K,n})])
        )
    }
    {
        \left\lfloor\frac{t_k}{sn}\right\rfloor
    }\\
    =&\frac
    {\#(M_{n,s}\cap sn\mathbb{N} \cap[0,t_k))}
    {\left\lfloor\frac{t_k}{sn}\right\rfloor}\\
    =&\frac
    {\zeta_{t_k}(M_{n,s}\cap sn\mathbb{N})}
    {\frac{t_k}{sn}}\cdot
    \frac
    {\frac{t_k}{sn}}
    {\left\lfloor\frac{t_k}{sn}\right\rfloor}.
  \end{aligned}
\end{equation}
Eqs (\ref{e:6.57}) and (\ref{e:6.56}) lead to
\begin{equation}\label{e:6.58}
  \lim_{k\rightarrow\infty}\frac
    {
        \zeta_{\left\lfloor\frac{t_k}{sn}\right\rfloor}
        (
            N_{\sigma_{K^{2sn}}}(\tau_{K^2,sn}(x),[\tau_{K^{2n}}(F^s_{K,n})])
        )
    }
    {
        \left\lfloor\frac{t_k}{sn}\right\rfloor
    }=1-q.
\end{equation}

Write
$$K_{sn,0}=E_{K,sn},\quad K_{sn,1}=\tau_{K^{2n}}(F_{K,n}^s),\quad K_{sn,2}=\{0,1,\cdots K^{2sn}-1\}\setminus (K_{sn,0}\sqcup K_{sn,1})$$
and
$$\mathcal{K}_{sn}=(K_{sn,0},K_{sn,1},K_{sn,2}).$$
Let
$$r=(q,1-q,0).$$
By (\ref{e:6.50}) and (\ref{e:6.58}), $\tau_{K^2,sn}(x)\in V_{\mathcal{K}_{sn},r}$. As $x\in X_n$ was arbitrary,
\begin{equation}\label{e:6.59}
  \tau_{K^2,sn}(X_n)\subset V_{\mathcal{K}_{sn},r}.
\end{equation}
Using Lemma \ref{detosx} and Lemma \ref{detoso} for (\ref{e:6.59}) we have
$$
\begin{aligned}
  &\dim_HX_n\\
  \leq&g_{\mathcal{K}_{sn}}(r)\\
  =&\frac{-q\ln\frac{q}{\#(E_{K,sn})}-(1-q)\ln\frac{1-q}{\#(F^s_{K,n})}}
  {\ln K^{2sn}}\\
  =&\frac{-q\ln\frac{q}{K^{sn}}-(1-q)\ln\frac{1-q}{(K^{2n}-K^n)^s}}
  {\ln K^{2sn}}\\
  =&\frac{-q\ln q+q\ln K^{sn}-(1-q)\ln(1-q)+(1-q)\ln (K^{2n}-K^n)^s}
  {\ln K^{2sn}}
\end{aligned}
$$
Then
$$
\begin{aligned}
  &\dim_HX_n-(1-\frac{q}{2})\\
  \leq&\frac{-q\ln q+q\ln K^{sn}-(1-q)\ln(1-q)+(1-q)\ln (K^{2n}-K^n)^s}
  {\ln K^{2sn}}-(1-\frac{q}{2})\\
  =&\frac{-q\ln q-(1-q)\ln(1-q)+(1-q)\ln (K^{2n}-K^n)^s-(1-q)\ln K^{2sn}}{\ln K^{2sn}}\\
  =&\frac{-q\ln q-(1-q)\ln(1-q)+(1-q)s\ln (1-K^{-n})}{\ln K^{2sn}}\\
  <&0\text{ for large }s.
\end{aligned}
$$
Then $\dim_HX_n<1-\frac{q}{2}$ and thus $\mathscr{H}^{1-\frac{q}{2}}(X_n)=0$.

Now
$$\mathscr{H}^{1-\frac{q}{2}}(X)=\mathscr{H}^{1-\frac{q}{2}}(\bigcup_{n\geq 0}X_n)
\leq\sum_{n\geq 0}\mathscr{H}^{1-\frac{q}{2}}(X_n)=0.$$
\end{proof}

\begin{lemma}\label{l:6.7}
For $0\leq p\leq 1$, $\mathscr{H}^1(D_{\sigma_K}([p,1]))=+\infty$.
\end{lemma}
\begin{proof}
Let $0\leq p\leq1$. Pick $M\in \mathcal{M}([p,1])$.
Let $N=2M+\{0,1\}$. Then $N\in \mathcal{M}([p,1])$.
Define $X\subset \Sigma_{K^2}$ by
$$
\begin{aligned}
  X&=\{x\in\Sigma_{K^2}:x|_{[2i,2i+2)}\in E_K^2\text{ for }i\in M\text{ and }x|_{[2i,2i+2)}\in F_{K,2}\text{ for }i\in M^c\}\\
  &=\{\Phi_{N,x}(y):x\in E_K^\mathbb{N},\,y\in F_{K,2}^\mathbb{N}\}.
\end{aligned}
$$
Since $\pi_K^{-1}(E_K^\mathbb{N})=\delta_{\Sigma_K}$ and $\pi_K^{-1}(F_{K,2}^\mathbb{N})=\text{Dist}(\sigma_K)$, by Lemma \ref{dstowv}, $\pi_K^{-1}(X)\subset D_{\sigma_K}([p,1])$.
Then, by Lemma \ref{detosw}, to prove $\mathscr{H}^1(D_{\sigma_K}([p,1]))=+\infty$ it is enough to prove $\mathscr{H}^{\frac{1}{2}}(X)=+\infty$.
Let
$$
\begin{aligned}
  Y&=\tau_{K^2,2}(X)\\
  &=\{x\in\Sigma_{K^4}:x_i\in \tau_{K^2}(E_{K,2})\text{ for }i\in M\text{ and }x_i\in \tau_{K^2}(F_{K,2})\text{ for }i\in M^c\}.
\end{aligned}
$$
By Lemma \ref{detosx}, to prove $\mathscr{H}^{\frac{1}{2}}(X)=+\infty$ it is enough to prove $\mathscr{H}^{\frac{1}{2}}(Y)=+\infty$.

Let
$$Z=(\tau_{K^2}(E_{K,2}))^\mathbb{N}\subset \Sigma_{K^4}.$$
Define $\phi:\{0,1,\cdots,K^2-1\}\rightarrow\tau_{K^2}(E_{K,2})$ and $\Phi:\Sigma_{K^2}\rightarrow Z$ by
$$\phi(i+Kj)=(i+Ki)+K^2(j+Kj),\,\,i,j\in\{0,1,\cdots,K-1\}$$
and
$$\Phi(x)=(\phi(x_i))_{i\geq 0},\,\,x\in \Sigma_{K^2}.$$
Then $\Phi$ is a bijection from $\Sigma_{K^2}$ to $Z$ with
$$\rho(\Phi(x),\Phi(y))=(\rho(x,y))^2,\,x,y\in\Sigma_{K^2}.$$
So
$$\mathscr{H}^\frac{1}{2}(Z)=\mathscr{H}^\frac{1}{2}(\Sigma_{K^2})=1.$$

Define
$$\Psi=\{(\psi_i)_{i\in M^c}:\text{ each $\psi_i$ is an injection from $\tau_{K^2}(E_{K,2})$ into }\tau_{K^2}(F_{K,2})\}.$$
For $\psi\in\Psi$ define
$$Y_{\psi}=\{x\in Y:x_i\in\psi_i(\tau_{K^2}(E_{K,2}))\text{ for }i\in M^2\}.$$
and define $T_\psi:Z\rightarrow Y_\psi$ by
$$(T_\psi(x))_i=x_i\text{ for }i\in M\text{ and }(T_\psi(x))_i=\psi_i(x_i)\text{ for }i\in M^2.$$
Then $T_\psi$ is an isometry between $Z$ and $Y_\psi$. So
$$\mathscr{H}^\frac{1}{2}(Y_\psi)=\mathscr{H}^\frac{1}{2}(Z)=1.$$
Since
$$\frac{\#(\tau_{K^2}(F_{K,2}))}{\#(\tau_{K^2}(E_{K,2}))}
=\frac{\#(F_{K,2})}{\#(E_{K,2})}
=\frac{K^4-K^2}{K^2}=K^2-1\geq 3,$$
then there is an uncountable set $\Psi_0\subset\Psi$ such that the sets $Y_\psi$ and $\psi\in\Psi_0$ are pairwise disjoint. Since each $Y_\psi$ is compact and thus $\mathscr{H}^\frac{1}{2}$ measurable,
$$\mathscr{H}^\frac{1}{2}(Y)\geq \sum_{\psi\in\Psi_0}\mathscr{H}^\frac{1}{2}(Y_\psi)=+\infty.$$

\end{proof}

We sum up Lemma \ref{l:6.6} and Lemma \ref{l:6.7} into the following theorem.
\begin{theorem}
For $0\leq p\leq q<1$, $\mathscr{H}^{2-q}(D_{\sigma_K}([p,q]))=0$. For $0\leq p\leq 1$, $\mathscr{H}^1(D_{\sigma_K}([p,1]))=+\infty$.
\end{theorem}

\section*{Acknowledgements}
The second author was supported by NNSF of China (11671208 and 11431012). We would like to express our gratitude to Tianyuan Mathematical Center in Southwest China, Sichuan University and Southwest Jiaotong University for their support and hospitality.

\end{document}